\documentclass[11pt,letterpaper]{article}

\usepackage{amsmath}
\usepackage{amssymb}
\usepackage{amsthm}
\usepackage[small,margin=3mm]{caption}
\usepackage{dsfont}
\usepackage[hmargin=27mm,top=30mm,bottom=35mm]{geometry}
\usepackage{graphicx}
\usepackage[latin1]{inputenc}
\usepackage{mathrsfs}
\usepackage{psfrag}
\usepackage[active]{srcltx}
\usepackage{url}
\usepackage{verbatim}

\newcommand{\dd}{{\mathrm d}}
\newcommand{\E}{\mathds{E}}
\newcommand{\N}{\mathds{N}}
\newcommand{\Z}{\mathds{Z}}
\newcommand{\R}{\mathds{R}}

\newcommand{\I}{\mathds{1}}

\theoremstyle{plain}
\newtheorem{theorem}{Theorem}[section]
\newtheorem{proposition}{Proposition}[section]
\newtheorem{lemma}{Lemma}[section]
\newtheorem{corollary}{Corollary}[section]
\newtheorem{claim}{Claim}[section]
\theoremstyle{definition}
\theoremstyle{remark}

\numberwithin{equation}{section}

\begin{document}

\title{The Discrete and Continuum Broken Line Process}
\author{Leonardo T. Rolla, Vladas Sidoravicius,\\
        Donatas Surgailis, Maria E. Vares \\ {\small 
        ENS-Paris, CWI and IMPA, IMI, CBPF
}}
\date{February 26, 2010}
\maketitle

\begin{abstract}
In this work we introduce the discrete-space broken line process
(with discrete and continues parameter values)
and derive some of its properties.
We explore polygonal Markov fields techniques developed by Arak-Surgailis.
The discrete version is presented first and a natural continuum generalization
to a continuous object living on the discrete lattice
is then proposed and studied.
The broken lines also resemble the Young diagram and the Hammersley process
and are useful for computing last passage percolation values and finding maximal
oriented paths.
For a class of passage time distributions there is a family of boundary conditions that make the process stationary and self-dual.
For such distributions there is a law of large numbers and the process extends to the infinite lattice.
A proof of Burke's theorem emerges from the construction.
We present a simple proof of the explicit law of large numbers for last passage percolation as an application.
Finally we show that the exponential and geometric distributions are the only non-trivial ones that yield self-duality.
\end{abstract}

This preprint has the same numbering of sections, equations, figures and theorems as the the published article
``\emph{Markov Process. Related Fields 16 (2010), no. 1, 79-116.}''

{\small AMS Subject Classifications: 60K35, 82B20, 60G60, 60G55. 

Key words: spatial random processes, Hammersley process, last passage
     percolation, time constant, broken line process.}

\section{Introduction}

The main goal of this work is the introduction and analysis of a process hereby called the \emph{continuum broken line process}, whose discrete analogue has been  considered for instance in~\cite{sidoravicius99}.
This process might be viewed as a generalization of the well known Hammersley process, as considered by Aldous and Diaconis~\cite{aldous95}, and also studied by Rost~\cite{rostXX}.
The Hammersley process fits into the broad class of polygonal Markov fields, which have been promoted by A.~N.~Kolmogorov in his later years, and which have been studied in a sequence of works by Arak and Surgailis~\cite{arak89,arak89b,arak89a,arak89c,arak90,arak91}.
Indeed, we find in writings of J.~Hammersley some indication that he also had thought about particular cases of such fields, without however developing the research in this direction. 

We may view these processes, which also resemble Young diagrams (see~\cite{fulton97}, \cite{pak06} and references therein), as useful tools for dealing with first/last passage problems, more specifically in the search for maximal oriented paths.
This is indeed our motivation to investigate them further.
The geometric approach undertaken in this work sheds different light into the basic problems.
It provides a clear geometric interpretation and transparent arguments for the law of large numbers (with concentration inequalities) for the asymptotic velocity of directed last passage percolation, results that have been obtained by various authors since the seminal article by Rost~\cite{rost81} (see for instance~\cite{balazs06,johansson00,seppalainen98}).
This is briefly reviewed and discussed in Section~\ref{sec1appl}.
Fluctuations, however, are still beyond the reach of these geometric techniques.

The \emph{discrete geometric broken line process} can be informally described as the following particle system with creation and annihilation.
For space-time coordinates $(t,x)\in \mathbb{Z}^2$ with $t+x$ even, the set of all such $(t,x)$ being denoted by $\tilde{\mathbb {Z}}^2$, consider independent random variables  $\xi_{t,x}$ which are geometrically distributed with parameter $0<\lambda^2<1$:   $P(\xi_{t,x}=k)=(1-\lambda^2)\lambda^{2k}$, $k=0,1,\dots.$ (The evolution 
takes place in discrete space-time coordinates, which is not the reason for the name). At each time-space point $(t,x) \in \tilde{\mathbb {Z}}^2$, $\xi_{t,x}$ pairs of particles are born with opposite velocities $\pm1$. The particles move with their constant velocities. When moving particles with opposite velocities collide, they annihilate each other. Because of the possibility of many particles sharing the same time-space point, one needs to fix an annihilation rule: we set the `oldest' particles  as those that get annihilated at each collision.

The discrete broken line process consists of the space-time trajectories of the particles. It is defined by first considering finite volume systems, on bounded time-space hexagonal domains $S$ in $\tilde{\mathbb {Z}}^2$.
Suitable (geometric) boundary conditions on the left boundary of $S$ (cf.\ Figure~\ref{fig1hexdomain}) allow the construction of the infinite volume system, due to the duality (or reversibility) that appears under the proper geometric distribution, which then yields a consistent set finite volume processes and allows to define the infinite volume process, stationary with respect to the translations in $\mathbb{Z}^2$.
The broken lines follow the trajectories of the particles; when the particle is annihilated it twists and follows the backward path of the particle responsible for the annihilation; when the particle is born it twists again, following the forward path of the pairing particle and so on.
This is detailed in Section~\ref{sec1discretebl}.

For the \emph{continuum broken line process}, to be examined in Section~\ref{sec1continuumbl}, instead of having a certain number of (pairs of) particles as above, we speak of a \emph{mass} $\xi$, where $\xi \in \mathbb {R}_+$.
The continuum broken lines are not necessarily determined by sequentially associating pairs of `particles' as in Section~\ref{sec1discretebl}, but instead we need to associate mass, and this cannot be understood by simply following the mass associations, as the mass may branch as well.
This brings additional difficulties to the construction itself, which can be dealt with by looking at finite-size broken lines.
That this approach is sufficient for the description and the applications one has in mind is clarified through Theorem~\ref{theo1etaell} and the connection to the last passage percolation problem is made clear in Proposition~\ref{prop1llpp}.
Reversibility plays naturally a very important role.
(See also Section~4 in~\cite{balazs06}.)
Our results also show that the boundary conditions give no asymptotic contribution to the total flow of broken lines that cross a given domain when they have this suitable distribution.

The characterization of geometric and exponential laws as the reversible measures in the broken line process is interesting by itself. Also, an alternative proof for Burke's theorem emerges from the analysis in this work. 

The paper is structured as follows: in Section~\ref{sec1discretebl} we discuss the simpler discrete process. The main part of the work is carried out in Section~\ref{sec1continuumbl}, where the continuum process is defined precisely and the main results of the paper are proven. Section~\ref{sec1appl} is devoted to the connection with first/last passage percolation; we first make the connection clear, and them present the proof of the large deviation estimates.
The Appendices contain detailed proofs of facts that should be clear from the picture.

\section{The geometric broken line process}
\label{sec1discretebl}

We start by describing the evolution of a particle system with discrete time step.
This system is defined on certain polygonal finite space-time domains.
Seen as a non-homogeneous Markov chain, its evolution is self-dual and consistent.
As a consequence, the model can be extended to the whole lattice.
The broken lines are defined from the space-time trajectories of
the particles.

Everything presented in this section will be generalized afterwards but we found
it instructive to introduce the discrete geometric broken lines beforehand.

\subsection{Evolution of a particle system}
\label{sec1evolution}

We consider a \emph{hexagonal domain} $S \subset \tilde {\Z}^2$, i.e., a domain
of the form
\begin{equation*}
  S = \bigl\{y=(t,x)\in \tilde {\Z}^2: t_0 \leqslant t\leqslant t_1,
  x_{t,-} \leqslant x\leqslant x_{t,+} \bigr\},
\end{equation*}
where $t_0, t_1 \in {\Z}, t_0 < t_1 $ are given points, and $x_{t,-} \leqslant
x_{t,+}, t_0 \leqslant t\leqslant t_1$ are paths in $\tilde {\Z}^2 $ such that
for some $t_0 \leqslant t_{01}^\pm \leqslant t_1$, $ x_{t+1, \pm} - x_{t,\pm} =
\pm 1 \,(t_0 \leqslant t < t^\pm_{01}) $ and $x_{t+1,\pm} - x_{t,\pm} = \mp 1
\,(t_{01}^\pm\leqslant t < t_1)$, as in Figure~\ref{fig1hexdomain}.
\begin{figure}[!htb]
{
\small
\psfrag{x}{$x$}
\psfrag{t01-}{$t_{01}^-$}
\psfrag{t01+}{$t_{01}^+$}
\psfrag{t0}{$t_{0}$}
\psfrag{t1}{$t_{1}$}
\psfrag{t}{$t$}
 \centering
 \includegraphics[scale=.3]{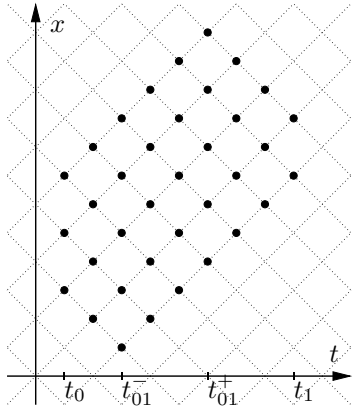}
\caption{
a hexagonal domain.}
\label{fig1hexdomain}
}
\end{figure}
Let $\partial_\pm S = \bigl\{(t,x) \in S: t_0\leqslant t\leqslant t^\pm_{01}, x= x_{t,\pm} \mbox{ or } t=t_0 \bigr\}$, so $\partial_+ S$, $\partial_- S$, and $\partial_+ S \cap \partial_- S$ are the northwest, southwest, and west boundaries, respectively.

For any point $y \in S$, let $\eta^+_{y}, \eta^-_{y} $ denote the
numbers of ascending (moving with velocity $+1$)  and descending
(moving with velocity $-1$) particles which `leave' $y$, and $\zeta^+_{y},
\zeta^-_{y} $ the respective numbers of ascending and descending particles
which `come' to $y$.
Each pair of incoming particles with opposite
velocities is annihilated, and a certain number $\xi_{y}\in\{0,1,2 \dots\} $ of
\emph{pairs} of particles with opposite velocities are born at $y$, that is,
\begin{equation}
\label{eq12.3}
\eta^+_{y} = \xi_{y} + \bigl[\zeta^+_{y}- \zeta^-_{y}\bigr]^+,
\quad
\eta^-_{y} = \xi_{y} + \bigl[\zeta^-_{y}- \zeta^+_{y}\bigr]^+,
\end{equation}
which implies
\begin{equation*}
 \zeta^+_{y} - \zeta^-_{y} =
 \eta^+_{y} - \eta^-_{y},
\end{equation*}
so at each point $y\in S$, particles can be created or killed by pairs only.
Moreover,
\begin{equation}
\label{eq12.4}
\zeta^\pm_{t,x} = \eta^\pm_{t-1,x\mp 1}, \quad y \in
S\backslash
\partial_\mp S,
\end{equation}
as all transformations of particles may occur only at lattice points
$y \in S$.
Put
\begin{eqnarray*}
& &
 \eta = (\eta^+_{y}, \eta^-_{y}: y \in S),
\\ & &
 \xi = (\xi_{y}: y \in S),
\\ & &
 \zeta^\circ_\pm=(\zeta_y^\pm:y\in\partial_\mp S),
\\ & &
 \zeta^\circ=(\zeta^\circ_+,\zeta^\circ_-).
\end{eqnarray*}
We may regard $\eta $ as the configuration in $S$ and $\zeta^\circ $ as the boundary data.
Notice that there is a 1-1 correspondence between $(\zeta^\circ, \xi ) $
and $(\zeta^\circ, \eta)$, where $\zeta^\pm_{y}, \eta^\pm_{y}, \xi_{y} $
are related by~(\ref{eq12.3})-(\ref{eq12.4}).
Note that~(\ref{eq12.3})~and~(\ref{eq12.4}) imply
\begin{equation}
\label{eq12.5}
\eta^+_{t,x} - \eta^-_{t,x} = \eta^+_{t-1,x-1}
- \eta^-_{t-1,x+1}, \quad
(t,x) \in S\backslash (\partial_+S \cup
\partial_- S).
\end{equation}

The probability measure $P_S$ corresponding to the evolution of this
particle system in the hexagonal domain $S$ can now be defined as follows.
Let $0<\lambda < 1$ be a parameter.
Assume that all $\zeta^\pm_{y}$'s in $\zeta^\circ $ are independent and
Geom($\lambda$)-distributed, i.e.
\[
P(\zeta^\pm_{y} =k) = (1-\lambda) \lambda^k, \quad k=0,1, \dots,
\quad y \in \partial_\mp S.
\]
Assume also that all $\xi_{y}$'s in $\xi $ are independent and
Geom($\lambda^2$)-distributed and, moreover, $\zeta^\circ $ and
$\xi $ are independent.
In other words, $\zeta^\pm_{y}, y \in \partial_\mp S$ (the number of particles which `immigrate' to $S$ through its left boundary $\partial_+S \cup \partial_-S $) are i.i.d.\ Geom($\lambda$)-distributed, and $\xi_{y}, y\in S$ (the number of pairs of particles with opposite velocities born inside $S$) are i.i.d.\ Geom($\lambda^2$)-distributed, independent of the $\zeta^\pm_{y}$'s.
During the evolution, the born particles move with constant velocities $+1$ or $-1$ until they collide at some lattice point $y \in S$, after which the colliding particles die (annihilate).
Let $P_S $ denote the resulting distribution of $(\zeta^\circ,\eta)$.

\subsection{An equivalent description for the evolution of the particles}

Below we provide another description of the evolution of the particle systems
introduced above.
It will be more convenient to derive some properties of the evolution using this
description.
Let $\Sigma_S $ be the set of all configurations $\eta$, where $\eta^\pm_{y}$
take values $0,1,\dots $ and satisfy relations~(\ref{eq12.5}), let $\Sigma_{\partial_\circ S} $ be the set of all configurations $\zeta^\circ$.
For $\zeta^\circ \in \Sigma_{\partial_\circ S}$, let $\Sigma_{S|\zeta^\circ}$ denote the class of all $\eta \in \Sigma_S$ that satisfy
\begin{equation}
\label{eq12.6}
\eta^+_{t,x} - \eta^-_{t,x}
= \begin{cases}
 \zeta^+_{t,x}- \eta^-_{t-1,x+1}, & \mbox{ if } (t,x)\in \partial_-S \backslash
 \partial_+S \\
 \eta^+_{t-1,x-1}- \zeta^-_{t,x}, & \mbox{ if } (t,x)\in \partial_+S \backslash
 \partial_-S \\
 \zeta^+_{t,x}- \zeta^-_{t,x},    & \mbox{ if } (t,x)\in \partial_+S \cap \partial_-S
.
 \end{cases}
\end{equation}
Let $\bar \Sigma_S $ be the set of all configurations $\bar \eta = (\zeta^\circ, \eta)$ with $\zeta^\circ \in {\Sigma}_{\partial^\circ S}$ and $\eta \in \Sigma_{S|\zeta^\circ}$.

We shall define a probability measure  $Q_{S|\zeta^\circ}(\eta)$ on
$\Sigma_{S|\zeta^\circ} $ as a (non homogeneous)
Markov chain whose values at each time $t$ are restrictions of a configuration
$\eta = \bigl(\eta^+_{u,x}, \eta^-_{u,x}: (u,x) \in S\bigr) \in
{\Sigma}_{S|\zeta^\circ}$ on $t=u$. In other words, $\eta_t = (\eta^+_{t,x},
\eta^-_{t,x}: x\in S_t)$, where $S_t = \bigl\{x: (t,x) \in S\bigr\}$.

The transition probabilities of the Markov chain are defined as follows (for
simplicity of notation, we assume $t^+_{01} = t^-_{01} = t_{01}$ below).

\noindent (i) At $t=t_0$, the distribution of
$\eta_{t_0} = (\eta^+_{t_0,x}, \eta^-_{t_0,x}: x \in S_{t_0} )$
depends only on $\zeta_{t_0} = (\zeta^+_{t_0,x}, \zeta^-_{t_0,x}:
x \in S_{t_0} ) $ and is given by
\begin{equation}
\label{eq12.7}
Q_{S|\zeta^\circ}(\eta_{t_0}|\zeta_{t_0}) = \prod_{x \in S_{t_0}}
q(\eta^+_{t_0,x}, \eta^-_{t_0,x}| \zeta^+_{t_0,x}, \zeta^-_{t_0,x}),
\end{equation}
where
\begin{equation}
\label{eq12.8}
q(n^+, n^-| m^+, m^-) = Z^{-1}_{m^+,m^-} \lambda^{n^++ n^-}
\delta_{\{n^+ -n^- = m^+-m^-\}},
\end{equation}
$n^\pm, m^\pm =0,1,\dots,\,
0< \lambda <1 $ is a parameter, and
\begin{equation}
\label{eq12.9}
Z_{m^+,m^-} = \sum_{n^+,n^-=0}^\infty \lambda^{n^++n^-}
\delta_{\{n^+-n^- = m^+-m^-\}}
= \lambda^{|m^+-m^-|}/(1-\lambda^2).
\end{equation}

\noindent (ii) Let $t_0 < t \leqslant t_{01}$.
The distribution of
$\eta_t = (\eta^+_{t,x}, \eta^-_{t,x}, x\in S_t) $
depends only on $\eta_{t-1} =
(\eta^+_{t-1,x}, \eta^-_{t-1,x}, x\in
S_{t-1}) $ and $ \zeta_{t,+} = \zeta^-_{t, x_{t,+}},
\zeta_{t,-} = \zeta^+_{t,x_{t,-}} $, according to the transition
probability
\begin{equation}
\label{eq12.10}
\begin{array}{rcl}
Q_{S|\zeta^\circ}(\eta_t|\eta_{t-1}, \zeta_{t,\pm})
&=& q(\eta^+_{t,x_{t,+}}, \eta^-_{t,x_{t,+}}|\eta^+_{t-1,x_{t-1,+}},
\zeta^-_{t,x_{t,+}})\\
&& \times\ \ q(\eta^+_{t,x_{t,-}}, \eta^-_{t,x_{t,-}}|\eta^+_{t-1,x_{t-1,-}},
\zeta^+_{t,x_{t,-}})\\
&& \times\ \ \prod_{x\in S_t, x\neq x_{t,\pm}} q(\eta^+_{t,x}, \eta^-_{t,x}|
\eta^+_{t-1,x-1}, \eta^-_{t-1,x+1}).
\end{array}
\end{equation}

\noindent (iii) Let $t_{01} < t \leqslant t_1$.
The distribution of
$\eta_t = (\eta^+_{t,x}, \eta^-_{t,x}, x\in S_t) $
depends only on $\eta_{t-1} =
(\eta^+_{t-1,x}, \eta^-_{t-1,x}, x\in
S_{t-1}) $ and
\begin{equation}
\label{eq12.11}
Q_{S|\zeta^\circ}(\eta_t|\eta_{t-1}) =
\prod_{x\in S_t} q(\eta^+_{t,x}, \eta^-_{t,x}|
\eta^+_{t-1,x-1}, \eta^-_{t-1,x+1}).
\end{equation}

Let $\Pi_\lambda (\zeta^\circ)$ be the product geometric distribution, i.e., all
random variables $\zeta^\pm_y, y\in \partial_\mp S$, are independent and Geom($\lambda $)-distributed.
Defining the probability measure $Q_S (\bar \eta) = Q_S (\zeta^\circ, \eta)$ by
\[
Q_S(\zeta^\circ, \eta ) =  \Pi_\lambda (\zeta^\circ)
Q_{S|\zeta^\circ}(\eta)
\]
we have $Q_S = P_S$.

\subsection{Self-duality (time reversibility)}
The transition probability~(\ref{eq12.8})
satisfies the following relation
\begin{equation}
\label{eq12.15}
\pi_\lambda (m^+) \pi_\lambda (m^-) q(n^+,n^-|m^+,m^-) =
\pi_\lambda (n^+) \pi_\lambda (n^-)
q(m^-,m^+|n^-,n^+),
\end{equation}
$m^\pm, n^\pm = 0,1,\dots, $ where $\pi_\lambda $
is Geom($\lambda$)-distribution.
A similar duality relation holds for transition probabilities of the Markov
chain $\eta_t, t_0 \leqslant t\leqslant t_1$.
For example, if $t_{01} < t\leqslant t_1$, then
from~(\ref{eq12.15}),(\ref{eq12.11}) it follows that
\begin{eqnarray*}
&\prod_{x\in S_t} \pi_\lambda (\eta^+_{t-1,x+1}) \pi_\lambda(\eta^-_{t-1,x+1})
q(\eta^+_{t,x},\eta^-_{t,x}| \eta^+_{t-1,x-1},
\eta^-_{t-1,x+1}) \\
&= \prod_{x\in S_t} \pi_\lambda (\eta^+_{t,x}) \pi_\lambda (\eta^-_{t,x})
q(\eta^-_{t-1,x+1},\eta^+_{t-1,x+1}| \eta^-_{t,x},
\eta^+_{t,x}).
\end{eqnarray*}
The above
duality implies that the construction of
$P_S (\bar \eta) = Q_S(\bar \eta) $ can be reversed in time. Namely, in the
`dual picture', $\eta^+_{y} $ is the number of
`descending' particles which `come' to $y$ from the right,
and $\zeta^-_{y}$ is the number of `ascending' particles which
`leave' $y$ in the same direction. A similar `dual interpretation'
can be given to $\eta^-_{y}, \zeta^+_{y}$.

In the dual construction, $\eta^\circ = (\eta^\pm_{y}: y \in \partial^\pm S)$ is the boundary
condition, where $\partial^\pm S$ are the northeast and southeast boundaries of $S$, respectively.
The probability measure $\hat Q_{S|\eta^\circ} (\zeta)$ is defined on the set
$\hat \Sigma_{S|\eta^\circ} $ of all configurations $\zeta = (\zeta^+_{y},
\zeta^-_{y}, y \in S)$ which satisfy conditions analogous to~(\ref{eq12.5}),(\ref{eq12.6}).
The definition of $\hat Q_{S|\eta^\circ} (\zeta)$ is completely analogous to
that of $Q_{S|\zeta^\circ}(\eta) $ and uses a Markov chain $\hat \eta_t, t_0
\leqslant t\leqslant t_1$ run in the time reversed direction, whose transition
probabilities are analogous to~(\ref{eq12.7}),(\ref{eq12.10}),(\ref{eq12.11}).
Then if $\Pi_\lambda (\eta^\circ) $ is the product geometric distribution on
configurations $\zeta^\circ$, we set
$$
\hat Q_S (\zeta, \eta^\circ) = \Pi_\lambda (\eta^\circ)
\hat Q_{S|\eta^\circ}(\zeta)
$$
and obtain the equality  for the probability measures on $\bar \Sigma_S$
\begin{equation}
\label{eq12.16}
\hat P_S = \hat Q_S
= Q_S = P_S.
\end{equation}

\subsection{Consistency}

Let $S' \subset S''$ be two bounded hexagonal domains in $\tilde {\Z}^2 $, and
let $P_{S'}, P_{S''} $ be the probability distributions on the
configuration spaces $\bar \Sigma_{S'}, \bar \Sigma_{S''}$, respectively, as
defined above.
The probability measure $P_{S''}$ on $\bar \Sigma_{S''}$ induces a probability
measure $P_{S''|S'}$ on $\bar \Sigma_{S'}$ which is the $P_{S''}-$ distribution
of the \emph{restricted process} $\bar \eta_{t,x}, (t,x) \in S$.
Then the following \emph{consistency property} is true:
\begin{equation}
\label{eq12.17}
P_{S''|S'} = P_{S'}.
\end{equation}
The proof of~(\ref{eq12.17}) uses~(\ref{eq12.16}) and the argument in Arak and
Surgailis~\cite[Theorem~4.1]{arak89}.
Let $\chi_\uparrow, \chi_\nearrow, \chi_\searrow $ denote a vertical line, an
ascending line, a descending line, respectively, in $\tilde {\Z}^2 $, with
the last two having slopes $\pi/4, -\pi/4$, respectively.
Any such line $\chi$ partitions $\tilde {\Z}^2 $ into the left part $\tilde
{\Z}^2_{\chi,-} = \bigl\{(t,x)\in \tilde {\Z}^2: t\leqslant t', x\leqslant x'
\mbox{ for some } (t',x') \in \chi \bigr\}$ (which contains the line $\chi $
itself) and the right part $\tilde {\Z}^2_{\chi,+} = \tilde {\Z}^2 \backslash
\tilde {\Z}^2_{\chi,-}$.
Note it suffices to show~(\ref{eq12.17}) for
$$
S' = S'' \cap \tilde {\Z}^2_{\chi,\pm},
$$
for any line $\chi = \chi_\uparrow, \chi_\nearrow, \chi_\searrow $ of the above
types.

When $S'= S''\cap \tilde {\Z}^2_{\chi_\uparrow,-}$
(\ref{eq12.17}) follows from the construction: in this case,
$P_{S''|S'} $ is nothing else but the evolution $\eta_t, t_0\leqslant t\leqslant
t'$ observed up to the moment $t'\leqslant t_1, (x,t') \in \chi_\uparrow $, and
therefore coincides with  $P_{S'}$.

The case $S' = S''\cap \tilde {\Z}^2_{\chi_\nearrow,-} $ follows by the
observation that $\bar \eta_{t,x}, (t,x) \in \tilde S\cap \tilde
{\Z}^2_{\chi_\nearrow,+} $ do not participate in the definition of the
probability of $\bar \eta_{t,x}, (t,x) \in S\cap \tilde {\Z}^2_{\chi_\nearrow,
-}$: the evolution of the particles after they exit through $\chi_\nearrow $ has
no effect on the evolution before they exit this line.
The case $S' = S''\cap \tilde {\Z}^2_{\chi_\searrow, -} $ is analogous.

The remaining cases are $S' = S''\cap \tilde {\Z}^2_{\chi,+},\,\chi = \chi_\uparrow, \chi_\nearrow, \chi_\searrow $.
However, the `reversibility'~(\ref{eq12.16}) allows to exchange the right and
left directions by replacing $P_S$ by the `reversed' process $\hat P_S$ and
thus reducing the problem to the cases considered above.

By consistency~(\ref{eq12.17}), the evolution of particles defined in finite
hexagonal domains can be extended to the evolution on the whole lattice $\tilde
{\Z}^2 $.
Let $\Sigma= \Sigma_{\tilde {\Z}^2} $ be the set of all configurations $\eta =
(\eta^+_{t,x}, \eta^-_{t,x}, (t,x) \in \tilde {\Z}^2 \}$
satisfying~(\ref{eq12.5}) for each $(t,x) \in \tilde {\Z}^2 $.
Then there exists a (unique) probability measure $P= P_{\tilde {\Z}^2} $ on $\Sigma$ whose restriction
$P_{|S}$ to an arbitrary hexagonal domain $S $ coincides with $P_S$:
$$
P_{|S} = P_S.
$$
Furthermore, $P$ is invariant with respect to translations of $\tilde {\Z}^2 $.

\subsection{Discrete broken line process}

In this section we shall describe the construction of the broken line process in a finite discrete hexagonal domain $S \subset \tilde {\Z}^2 $, similar to the construction for the continuous Poisson model.

When $|y -y'| = \sqrt{2}$, let $e=\langle y,y'\rangle$ denote the edge between $y$ and $y'$ and let $\E(\tilde\Z^2)$ be the set of all edges in $\tilde\Z^2$.
Given a hexagonal domain $S\subset \tilde {\Z}^2 $, let
\begin{equation}
\label{eq1sbar}
\begin{array}{rcl}
\bar S &=& S \cup \bigl\{y \in \tilde {\Z}^2: \langle y,y'\rangle \in \E(\tilde\Z^2) \mbox{ for some }y' \in S \bigr\}
,\\
\partial\bar S&=&\bar S\backslash S
,\\
\E(\bar S)&=&\bigl\{\langle y,y'\rangle\in\E(\tilde\Z^2):y\in S\bigr\}
,
\end{array}
\end{equation}
so $\bar S$ consists of all points of $S$ plus the neighboring points and $\E (\bar S)$ is the set of edges inside $\bar S$.

Given $y=(t,x) \in S$, we denote by $e^\nearrow_{y}, e^\nwarrow_{y}, e^\searrow_{y}, e^\swarrow_{y} $ the edge of $\tilde {\Z}^2 $ incident with $y$ and lying in northeast, northwest, southeast, southwest direction, respectively, so that $e^\nearrow_{t,x} = e^\swarrow_{t+1,x+1}, e^\nwarrow_{t,x} = e^\searrow_{t-1,x+1}$.
We shall consider configurations $\bar \eta \in \bar \Sigma_S$ as
\begin{equation}
\label{eq12.20}
\bar \eta = \bigl(\eta (e), e\in \E (\bar S)\bigr),
\end{equation}
where $\eta (e)$ is the number of particles which pass through the edge $e$; $\eta (e) = \eta^+_{y}, \eta^-_{y}, \zeta^+_{y}, \zeta^-_{y} $ for $e = e^\nearrow_{y}, e^\swarrow_{y}, e^\swarrow_{y}, e^\nwarrow_{y}$,
respectively.
This field $\bar\eta$ is what we shall call \emph{flow field}.

With each configuration $\bar \eta = (\zeta^\circ, \eta) \in \bar \Sigma_S$, one can associate a finite partially ordered system $\{\gamma_j, j=1, \dots, M \} $ of broken lines in $\bar S$, such that for any point $y \in \tilde {\Z}^2 $, the relations
\begin{equation}
\nonumber
\nu (e) = \eta (e), \quad e\in \E(\bar S).
\end{equation}
hold, where $\nu (e)$ is the number of broken lines $\gamma_i$ which pass through a given edge $e$ of $\tilde {\Z}^2 $, and where $\eta^+_{y}, \eta^-_{y}, \zeta^+_{y}, \zeta^-_{y} $ are related by~(\ref{eq12.3})-(\ref{eq12.4}), and denote the respective numbers of outgoing and incoming particles to a given site $y \in S$, as in Section~\ref{sec1evolution}.

A tuple $\ell$ of the form $(y_0,e_1,y_1,e_2,y_2,\dots,e_n,y_n)$ will be called a \emph{broken trace} if $\ell$ starts at the bottom part of $\partial\bar S$, remains in $S$, ends at the top part of $\partial\bar S$, and satisfies
\begin{equation}
\nonumber
\begin{array}{l}
e_i=\langle y_{i-1},y_i\rangle \\
x_i = x_{i-1}+1\\
t_i = t_{i-1}\pm 1
\end{array}.
\end{equation}

The broken lines for a given configuration $\bar \eta \in \bar {\E}_S$ can be constructed as follows.
For each particle passing through a given edge $e\in E(\bar S)$, we define a label $p\in\{1,\dots,\eta(e)\}$, which is interpreted as the \emph{relative age} of that particle among all particles which pass through the same edge.
The particle whose relative age is $p=1$ is the oldest and that with $p=\eta(e)$ is the youngest.

This way, we order coherently moving particles.
Any particle is characterized by a pair
\begin{equation}
\nonumber
  (e, p), \quad e\in \E(\bar S), \,\,p=1,2, \dots, \eta (e).
\end{equation}
We now define a relation between labeled particles $(e_1,p_1), (e_2, p_2)$ on two adjacent edges $e_1, e_2 \in \E(\bar S)$, which we denote by
\begin{equation*}
(e_1,p_1) \sim (e_2,p_2).
\end{equation*}
If relation $(e_1,p_1)\sim(e_2,p_2)$ holds, we say that they are \emph{associated}, or \emph{belong to the same generation}.
Suppose $e_1, e_2$ are any two edges $e^\nearrow_{y}, e^\nwarrow_{y}, e^\searrow_{y}, e^\swarrow_{y} $ incident with some $y \in S$.
The relation $(e_1,p_1) \sim (e_2,p_2)$ is defined in the following cases:

Case 1:
$e_1 = e_y^\swarrow$, $e_2 = e_y^\nwarrow$;
$p_1,p_2 \leqslant \eta(e_1)\wedge\eta(e_2)$

Case 2:
$e_1 = e_y^\searrow$, $e_2 = e_y^\nwarrow$;
$\eta(e_2) > \eta (e^\swarrow_{y})$, $p_2 > \eta (e^\swarrow_{y})$.

Case 3:
$e_1 = e_y^\swarrow$, $e_2 = e_y^\nearrow$;
$\eta(e_1) > \eta (e^\nwarrow_{y})$, $p_1 > \eta (e^\nwarrow_{y})$.

Case 4:
$e_1 = e_y^\searrow$, $e_2 = e_y^\nearrow$;
$\eta(e_i)-\xi_y < p_i \leqslant \eta(e_i)$.

Namely, $(e_1,p_1) \sim (e_2,p_2)$ holds if and only if

Case 1:
$p_2 = p_1$;

Case 2:
$p_1 = p_2 - \eta(e_y^\swarrow)$;

Case 3: 
$ p_2 = p_1 - \eta(e_y^\nwarrow)$;

Case 4:
$ \eta(e_1) - p_1 = \eta(e_2) - p_2$.

According to the above rules, associated particles are either the particles which annihilate each other at $y$ (Case 1), or the particles born at $y$ and moving into opposite directions (Case 4), or, as in Cases 2 and 3, we associate a younger incoming particle which is not killed at $y$ with the corresponding older outgoing particle (both of them move in the same direction).

We define a \emph{broken line} as a tuple $\gamma = (y_0,e_1,p_1,y_1,e_2,p_2,y_2,\dots,e_n,p_n,y_n)$, $n\geqslant1$, such that $\ell(\gamma) := (y_0,e_1,y_1,e_2,y_2,\dots,e_n,y_n)$ is a broken trace.
Now, we say that $\gamma = (y_0,e_1,p_1,y_1,e_2,p_2,y_2,\dots,e_n,p_n,y_n)$ is \emph{a broken line in a given configuration} $\bar \eta $ if any two adjacent edges of this path are associated, that is, they belong to the same generation.
More precisely, this condition means that $(e_{j},p_{j}) \sim (e_{j+1},p_{j+1})$ for $j=0,\dots,n$.

Let $\{\gamma_j, j=1, \dots, M \} $ be the family of all broken lines in a given configuration $\bar \eta $.
The corresponding family $\{\ell_j,j=1,\dots, M\} $ is well ordered by the relation $\preceq $ defined above, with the possible exception of broken lines that exit $S$ at its left or right boundaries, in which case two lines may not be ordered.
This follows from the definition of broken line and the fact that different pairs of associated particles cannot `cross' each other: if $ (e'_1, p'_1) \sim (e'_2, p'_2), \,(e''_1, p''_1) \sim (e''_2, p''_2) $ and $e'_1,e'_2, e''_1, e''_2 $ are all incident with the same vertex $y $, then the paths $(e'_1,e'_2) $ and $(e''_1, e''_2)$ cannot cross each other, in other words, they cannot lie on different lines intersecting at $y$.

\section{The continuum broken line process}
\label{sec1continuumbl}

In this section we present a natural generalization of the geometric broken line process, which we call the continuum broken line process.
In the continuum framework, instead of having a certain number $\xi\in\Z_+$ of pairs of particles being born at each site, we have a mass $\xi\in\R_+$.
In this case a broken line is described not only by the sites it occupies, but it also has some thickness.
The continuum broken lines are not determined by sequentially associating pairs $(e,p)$, we rather need to associate mass, and they cannot be understood by just following the mass associations over and over, as this mass may well branch due to the association rules.
For this reason we always consider broken lines of finite size, which suffices for the description of the process.

Some applications to last passage percolation are shown in Section~\ref{sec1appl}, where the proofs of Theorem~\ref{theo1exp} and Theorem~\ref{theo1llngeo} illustrate how this process can be useful.

Notation will often be abused in the sense that a different symbol can refer to an object of a given type, or to the object of that type that is obtained from other related objects, etc., but this should not give rise to any confusion.

\subsection
{Flow fields and self-duality}

Let $S$ be a fixed hexagonal domain in $\tilde\Z^2$ and recall the definitions from~\eqref{eq1sbar}.
We call $S$ a \emph{rectangular domain} if it is a degenerate hexagonal domain, i.e., $\#\partial_\pm S=\#\partial^\pm S=1$.
It is convenient to define,
for a rectangular domain $S$, the boundaries
$\partial_\pm\bar S = \bigl\{(t,x)\in\partial\bar S:(t+1,x\mp1)\in S\bigr\}$ and
$\partial^\pm\bar S = \bigl\{(t,x)\in\partial\bar S:(t-1,x\mp1)\in S\bigr\}$.

Consider $\{\xi_y\}_{y\in S}$, $\xi_y\geqslant0\ \forall {y\in S}$, the
\emph{particle birth process} in $S$.
$\zeta^\circ=(\zeta^\circ_+,\zeta^\circ_-)$ is the \emph{boundary condition}, or the \emph{particle flow entering} $S$.
One takes $\eta_y^\pm \geqslant0,\ y\in S$, the particle flow inside $S$ and exiting $S$, defined by~(\ref{eq12.3})-(\ref{eq12.4}).

Define $\bar\eta = \bigl\{\eta(e)\geqslant 0,\ e\in{\E}(\bar S)\bigr\}$,
called the \emph{flow field} in $\E(\bar S)$ associated with
$(\zeta^\circ,\xi)$.
As in the discrete case, there is a 1-1 correspondence between
$(\zeta^\circ,\xi)$, $(\zeta^\circ,\eta)$, and $\bar\eta$.
We shall write $\bar\eta(\zeta^\circ,\xi)$ to denote the flow field $\eta$ corresponding to the birth process $\xi$ and the flow $\zeta^\circ$ entering the boundary of $S$ as defined in~\eqref{eq12.20}, and analogously for $\bar\eta(\zeta^\circ,\eta)$.

In general, any nonnegative field $\bar\eta$ that satisfies
\begin{equation}
 \label{eq1lcons}
 n^+-n^-=m^+-m^-
\end{equation}
for all vertex $y$, where $n^+=\eta(e_y^\swarrow)$, $n^-=\eta(e_y^\nwarrow)$,
$m^+=\eta(e_y^\nearrow)$ and $m^-=\eta(e_y^\searrow)$, is a flow field.
Such a flow field may be defined either on a hexagonal domain $\E(\bar S)$ or on
all $\E(\tilde\Z^2)$.
In the former case there always exists a unique pair $(\zeta^\circ,\xi)$ such
that $\bar\eta=\bar\eta(\zeta^\circ,\xi)$.
In the latter case such an expression does not make sense, although it is possible
to determine $\xi(\bar\eta)$ for a given flow field $\bar\eta$.

\paragraph{Self-duality}

In the previous section we showed that the distribution $P_S$ of the configuration $\bar\eta$ was self-dual when the distributions of $\zeta^+$, $\zeta^-$, and $\xi$ imply (\ref{eq12.15}).

More generally, denoting by $\pi_1,\pi_2,\pi_3$ the distributions of $\zeta^+$, $\zeta^-$, and $\xi$, respectively, self-duality is determined by the following relation
\begin{equation}
\label{eq1lrev}
 \begin{array}{c}
  \pi_1(\dd m^+)\pi_2(\dd m^-)q(\dd n^+,\dd n^-|m^+,m^-) = \\ = \pi_1(\dd
n^+)\pi_2(\dd n^-)q(\dd m^+,\dd m^-|n^+,n^-),
 \end{array}
\end{equation}
where $q$ is given by $q(n^+ \wedge n^-=\dd t|m^+,m^-)=\pi_3(\dd t)$ and $q\left(\{(n^+,n^-):\mbox{(\ref{eq1lcons}) holds}\}\big|m^+,m^-\right)=1$.

Thus any triple $\pi_1,\pi_2,\pi_3$ that satisfies (\ref{eq1lrev}) will define a family of measures $\{P_S:S\mbox{ hexagonal domain}\}$ which is consistent, i.e., satisfies~(\ref{eq12.17}).
In this case $\bar\eta$ can be consistently extended to ${\E}(\tilde \Z^2)$.
As another consequence, if one takes $(\zeta^+_y)_{y\in\partial_-S}$ i.i.d.\ distributed as $\pi_1$, $(\zeta^-_y)_{y\in\partial_+S}$ i.i.d.\ having law $\pi_2$, and $(\xi_y)_{y\in S}$ i.i.d.\ with law $\pi_3$, then $(\eta_y^+)_{y\in\partial^+S}$ will be i.i.d.\ with law $\pi_1$ and independent of $(\eta_y^-)_{y\in\partial^-S}$, which will be i.i.d.\ with law $\pi_2$.

Below we characterize the distributions that satisfy~(\ref{eq1lrev}).
Since the particle system always evolves by keeping the relation (\ref{eq1lcons}), we can parametrize the corresponding hyperplane in $\R_+^4$ defining $T:\R_+^3\to\R_+^4$ by
\begin{equation}
\nonumber
(m^+,m^-,n^+,n^-)=T(r,s,t):=(r,s,t+[r-s]^+,t+[r-s]^-).
\end{equation}
The joint distribution of $(m^+,m^-,n^+,n^-)$ is given by the left-hand side of (\ref{eq1lrev}), and it can be obtained by taking the image $T\mu$, where $\mu = \pi_1 \times \pi_2 \times \pi_3$.
Now (\ref{eq1lrev}) just means that $T\mu$ is preserved by the operator $L$ in $\R_+^4$ given by $L(x,y,z,w)=(z,w,x,y)$.
This is equivalent to the fact that
\begin{equation}
 \label{eq1lmpres}
  R\mu = \mu,
\end{equation}
where $R=T^{-1}LT$.
Writing $R$ explicitly gives
\[
(r',s',t')=R(r,s,t):=(t+[r-s]^+,t+[r-s]^-,r\wedge s).
\]

It is straightforward to check that~(\ref{eq1lmpres}) is satisfied when $\pi_i=\exp(\alpha_i)$ with $\alpha_3=\alpha_1+\alpha_2$ or $\pi_i=\mathrm{Geom}(\lambda_i)$ with $\lambda_3=\lambda_1\lambda_2$, of which the geometric broken line process described in the previous section is a particular case.
The following theorem says that these are the only non-trivial examples.
Denote the support of a Borel probability $\pi$ by $\mathrm{supp}(\pi)=\bigcap\{U:{U\subseteq \R \mbox{ closed }, \pi(U)=1}\}$.
\begin{theorem}
 \label{theo2dualgeoexp}
 Let $\pi_1,\pi_2,\pi_3$ be probability distributions in $\R_+$, such that $\mu = \pi_1 \times
\pi_2 \times \pi_3$ satisfies~\eqref{eq1lmpres}.
 Take $\tilde U_i=\mathrm{supp}(\pi_i)$, $t_0=\min\tilde U_3$, $U_i=\tilde U_i-t_0$ and $F_i(t)= \pi_i\big\{[t_0 + t,\infty)\big\}$.
 Then one of the following holds:
\begin{enumerate}
 \item  $U_3=\{0\}$.
 \item $U_1=U_2=U_3=\R_+$, $F_i(t)=\exp(-\alpha_i t)$ for some $\alpha_1,\alpha_2\in(0,\infty)$ and $\alpha_3=\alpha_1+\alpha_2$.
 \item $U_1=U_2=U_3=c\Z_+$ for some $c>0$, $F_i(cn)=\lambda_i^n$ for some $\lambda_1,\lambda_2\in(0,1)$ and $\lambda_3=\lambda_1\lambda_2$.
\end{enumerate}
\end{theorem}
\begin{proof}

It follows from~\eqref{eq1lmpres} that $\min(\tilde U_1 \cup \tilde U_2)= t_0$, and so we can assume without loss of generality that $t_0=0$ and $\tilde U_i=U_i$.
If $U_1=\{0\}$ or $U_2=\{0\}$, by~\eqref{eq1lmpres} $U_3=\{0\}$ and we are done.
Otherwise take $u_1,u_2>0$ in $U_1$ and $U_2$ respectively, and write $u_0=u_1\wedge u_2>0$.
It follows from~\eqref{eq1lmpres} that $u_0\in U_3$.

We show that $U_1$ and $U_2$ contain $0$ and are unbounded;
this together with~\eqref{eq1lmpres} implies that $U_1=U_2=U_3$, which we denote by $U$.
As $0\in U_1\cup U_2$, let us assume that $0\in U_1$.
To see that $U_2$ is unbounded, notice that $(0,u_2,u_0)\mapsto (u_0,u_2+u_0,0)$, thus $u_2+u_0 \in U_2$, and again $u_2+2u_0\in U_2$, and so on.
To see that $0\in U_2$, notice that $(u_0,u_2,0)\mapsto(0,[u_2-u_0]^+,u_0)$, thus $[u_2-u_0]^+\in U_2$, and again $[u_2-2u_0]^+\in U_2$, and so on.
Once $0\in U_2$, the previous argument implies that $U_1$ is unbounded.

Now notice that $(U\cup -U)$ is a subgroup of $(\R,+)$.
Indeed, if $u,v\in U$, then it follows from~\eqref{eq1lmpres} that $|u - v|\in U$ as $(u,v,0)\in U^3$, and it also follows that $u+v\in U$ as $(u,0,v)\in U^3$.
But $U$ is closed, so either $U=\R_+$ or $U=c\Z_+$ and without loss of generality $c=1$.

The case $U=\Z_+$ is simpler.
Write $f_i(n)=F_i(n)-F_i(n+1)=\pi_i(\{n\})>0$.
From~\eqref{eq1lmpres} we have $f_1(s)F_2(r+s)f_3(0)=f_1(0)F_2(r)f_3(s)$, so $\frac{F_2(r+s)}{F_2(r)}=\frac{f_1(0)f_3(s)}{f_1(s)f_3(0)}$ does not depend on $r$.
Therefore $F_2$ is exponential, and by symmetry the same is true for $F_1$.
The result follows since $F_3(t)=F_1(t)F_2(t)$ by~\eqref{eq1lmpres}.

Finally suppose $U=\R_+$.
Write $f_i^\epsilon(s)=F_i(s)-F_i(s+\epsilon)=\pi_i[s,s+\epsilon)>0$.
There is a dense set in $\R_+$ where the $F_i$ are continuous, which we call good points.
Let $s,t>\epsilon>0$.
From~\eqref{eq1lmpres} we get
$
  f_1^\epsilon(s) F_2(r+s+\epsilon) f_3(0)
  \leqslant
  f_1^\epsilon(0) F_2(r) f_3(s)
  \leqslant
  f_1^\epsilon(s) F_2(r+s-\epsilon) f_3(0),
$
so
$
  \frac {F_2(r+s+\epsilon)}{F_2(r)}
  \leqslant
  \frac {f_1^\epsilon(0) f_3(s)}{f_1^\epsilon(s) f_3(0)}
  \leqslant
  \frac {F_2(r+s-\epsilon)}{F_2(r)}.
$
The ratio $\frac {F_2(r+s)}{F_2(r)}$ is thus independent of $r$ as long as $r+s$ is a good point.
As $F_2$ is a decreasing function, there are $\alpha_2>0$ and $q\in(0,1]$ such that $F_2(r)=q\exp(\alpha_2r)$ for $r>0$.
Analogously, there are $\alpha_1>0$ and $p\in(0,1]$ such that $F_1(r)=p\exp(\alpha_1r)$ for $r>0$.
Finally notice that the only possible singleton of $\pi_i$ is $0$, so $(1-p)(1-q)=\mu\{(r,s,t):r=s\}=\mu\{(r,s,t):t=0\}=1-pq$; whose unique solution in $(0,1]$ is $p=q=1$.
\end{proof}

\subsection{Construction of the broken lines}

Let $S$ be a fixed hexagonal domain in $\tilde\Z^2$.
For a given flow field $\bar \eta$ in $\E(\bar S)$, we shall define a process of lines whose elementary constituents, called \emph{atoms}, are pairs of the form $(e,p),e\in \E(\bar S), p\in\bigl(0,\eta(e)\bigr]$.

We associate two atoms and write $(e_1,p_1)\sim(e_2,p_2)$ according to the
following rules:

Case 1: $e_1 = e_y^\swarrow$, $e_2 = e_y^\nwarrow$;
$p_1\in\bigl(0,\eta(e_1)\wedge\eta(e_2)\bigr],\ p_2 = p_1$;

Case 2: $e_1 = e_y^\searrow$, $e_2 = e_y^\nwarrow$;
$p_2\in\bigl(\eta(e_y^\swarrow),\eta(e_2)\bigr],\ p_1 = p_2 -
\eta(e_y^\swarrow)$;

Case 3: $e_1 = e_y^\swarrow$, $e_2 = e_y^\nearrow$;
$p_1\in\bigl(\eta(e_y^\nwarrow),\eta(e_1)\bigr],\ p_2 = p_1 -
\eta(e_y^\nwarrow)$;

Case 4: $e_1 = e_y^\searrow$, $e_2 = e_y^\nearrow$;
$p_i\in\bigl(\eta(e_i)-\xi_y,\eta(e_i)\bigr],\ \eta(e_1) - p_1 = \eta(e_2) -
p_2$;

Notice that, given a vertex $y$, each atom standing at an edge incident to $y$
from above is associated with exactly one atom standing at another edge incident
to $y$ from below and vice versa.
Notice also and that `$\sim$' is not transitive.

By $J$ we will always denote an interval of the form $(p,p']$ and $|J|=p'-p$.
When $p\geqslant p'$, $J=\emptyset$ and $|J|=0$.
We associate two intervals of atoms standing at adjacent edges and write
$(e_1,J_1)\sim(e_2,J_2)$ if for all $p_1\in J_1$ there is $p_2\in J_2$ such that
$(e_1,p_1)\sim(e_2,p_2)$ and vice versa.
Notice that in this case $|J_1|=|J_2|$.

We define a \emph{broken line} as a tuple $\gamma =
(y_0,e_1,J_1,y_1,e_2,J_2,y_2,\dots,e_n,J_n,y_n)$, $n\geqslant1$, such that
\begin{equation}
 \label{eq1lbroktr}
\begin{array}{l}
e_i=\langle y_{i-1},y_i\rangle \\
x_i = x_{i-1}+1\\
t_i = t_{i-1}\pm 1
\end{array},
\end{equation}
and $y_i=(t_i,x_i)$ for $i=1,\dots,n$, and such that $|J_1|=|J_2|=\cdots=|J_n|$.
If $J_i=\emptyset$ we identify $\gamma=\emptyset$.

We define next the \emph{weight} of a broken line $w(\gamma)$.
If $\gamma\ne\emptyset$, we put $w(\gamma) = |J_1|>0$, otherwise we let
$w(\gamma)=0$.

For $\gamma = (y_0,e_1,J_1,y_1,e_2,J_2,y_2,\dots,e_n,J_n,y_n)$ we define what is
called its \emph{trace} by $\ell(\gamma) = (y_0,e_1,y_1,e_2,y_2,\dots,e_n,y_n)$.
In general, any tuple $\ell$ satisfying (\ref{eq1lbroktr}) will be called a
\emph{broken trace}.
Since either $(y_0,\dots,y_n)$ or $(e_1,\dots,e_n)$ are sufficient to determine
$\ell$, we shall refer to any of these representations without distinction.

The domain of a broken line $\gamma$ or a broken trace $\ell$ satisfying~\eqref{eq1lbroktr} is given by $D(\gamma)=D(\ell)=\{x_0,\dots,x_n\}$.
Since for each $x \in D(\gamma)$ there is a unique  $t$ such that $(t,x) \in \gamma$, we denote such $t$ by $t(x)$.
We set $I(\ell)=\{t(x)\colon x \in D(\ell)\}$.
It follows from~\eqref{eq1lbroktr} that $D(\ell)$ and $I(\ell)$ are convex.
(Notice that we abuse the symbol `$\in$' since $\gamma$ is a tuple instead of a set.)

We write $\ell\subseteq \bar S$ if $y_0\in S\cup\partial_- \bar S\cup\partial^-
\bar S$, $y_n\in S\cup\partial_+ \bar S\cup\partial^+ \bar S$ and
$y_1,\dots,y_{n-1}\in S$, or, equivalently, if $e_1,\dots,e_n\in\E(\bar S)$.
We say that $\ell$ \emph{crosses} $S$ if $\ell\subseteq \bar S$ and
$y_0,y_n\in\partial\bar S$. Let $C(S)=\{\ell\subseteq\bar S:\ell \mbox{ crosses
} S\}$.
We write $\gamma\subseteq\bar S$ if $\ell(\gamma)\subseteq\bar S$ and we say
that $\gamma$ \emph{crosses} $S$ if $\ell(\gamma)$ crosses $S$.

We say that the broken line $\gamma\subseteq\bar S$ is associated with a flow
field $\bar\eta$ defined on $\E(\bar S)$ if $(e_{i-1},J_{i-1})\sim(e_i,J_i)$ for
$i=1,\dots,n$. $B(\bar\eta)$ will denote the set of all broken lines associated
with $\bar\eta$.

For $V\subseteq\tilde\Z^2$, and $\ell=(y_0,\dots,y_n)$, define $\ell\cap V =
\{y_i:y_i\in V,i=0,\dots,n\}$.

For a given broken line $\gamma$ associated with the field
$\bar\eta(\zeta^\circ,\xi)$, its left corners $(t,x)$ correspond to part of the
particle birth $\xi_{t,x}$ at $(t,x)$.
Given a broken trace $\ell$, we denote by $L(\ell)$ the set of left corners of
$\ell$, i.e., the points $(t,x)\in \ell$ such that $(t+1,x\pm1)\in\ell$.
Also let $L(\gamma)=L(\ell(\gamma))$.

We define the fields $\xi(\ell)$ and $\xi(\gamma)$ in $\tilde\Z^2$ by
\begin{equation}
 \label{eq1lxidef}
 [\xi(\ell)]_y =\I_{L(\ell)}(y)
\end{equation}
and $\xi(\gamma) = w(\gamma) \I_{L(\gamma)}$.

In the same fashion, the extremal points of a broken line that crosses $S$
correspond to a particle flow entering or exiting $S$.
So we also define
$$
  [\zeta^+(\ell)]_{t,x} = 
  \begin{cases}
    1, & (x,t)\in \ell\cap S \mbox{ and } (t-1,x-1)\in\ell\cap\partial\bar S, \\ 
    0, & \mbox{otherwise},
  \end{cases}
$$
$$
  [\zeta^-(\ell)]_{t,x} = 
  \begin{cases}
   1, & (x,t)\in \ell\cap S \mbox{ and }(t-1,x+1)\in\ell\cap\partial\bar S,  \\ 
   0, & \mbox{otherwise},
  \end{cases}
$$
the corresponding $\zeta^\circ$, and denote it by $\zeta^\circ(\ell)$.
Define $\zeta^\circ(\gamma) = w(\gamma)\zeta^\circ(\ell(\gamma))$.
Also define $\bar\eta(\ell)$ by $\bar\eta(\ell)(e)=\I_{e\in\ell}$ and
$\bar\eta(\gamma)=w(\gamma)\bar\eta(\ell(\gamma))$.
Notice that by construction $\bar\eta(\ell)=\bar\eta\bigl(\zeta^\circ(\ell),\xi(\ell)\bigr)$ for $\ell\in
C(S)$, and analogously for $\gamma$.

Let $\bar\eta$ be given and fix some
$\ell=(y_0,e_1,y_1,e_2,y_2,\dots,e_n,y_n)\subseteq\bar S$.
There is one maximal broken line that has trace $\ell$ and is associated with
the field $\bar\eta$, which will be denoted $\gamma(\ell)$, the dependence on
$\bar\eta$ is omitted.
By this we mean that there exist unique $J_1, J_2, \dots,J_n$ such that
$\gamma(\ell)=(y_0,e_1,J_1,y_1,\dots,e_n,J_n,y_n) \in B(\bar\eta)$ and such that
any $p_1,p_2,\dots,p_n$ with property
$(e_1,p_1)\sim(e_2,p_2)\sim\cdots\sim(e_n,p_n)$ must satisfy $p_i\in J_i,\
i=1,\dots,n$.
The proof of this fact is shown in Appendix~\ref{sec1maximal}.

Let $w(\ell)$ denote the maximum weight of a broken line in $B(\bar\eta)$ that
has trace $\ell$, which is given by $w\bigl(\gamma(\ell)\bigr)$.
The dependence on the field $\bar\eta$ is omitted when it is clear which field
is being considered, otherwise we shall write $w_{\eta}(\ell)$.
Notice that $D\bigl(\gamma(\ell)\bigr) = D(\ell)$ if $w_\eta(\ell)>0$, and
$D\bigl(\gamma(\ell)\bigr) =\emptyset$ otherwise.

We write $\ell\subseteq\ell'$ if $D(\ell)\subseteq D(\ell')$ and $t(x)=t'(x)$ for all $x\in D(\ell)$.

Notice that $w(\ell_1) \geqslant w(\ell_2)$ when $\ell_1\subseteq\ell_2$.
This is due to the successive branching caused by birth/collision that makes
longer lines become thinner.
It is even possible that for
$\ell_1\subseteq\ell_2\subseteq\cdots\subseteq\ell_n\subseteq\cdots$ we have
$w(\ell_n)\downarrow0$ and the limiting object could have one atom but null
weight.

For broken traces $\ell,\ell'$, we write $\ell \succeq \ell'$ if $\ell$ is to
the right of $\ell'$.
This means that $t(x)\geqslant t'(x)$ for all $x\in D(\ell)\cap D(\ell')$ and
that $t(x)\geqslant t'(x')$ for some $x\in D(\ell),x'\in D(\ell')$.
(The last condition makes sense in case $D(\ell)\cap D(\ell')=\emptyset$.)
In general this relation is neither antisymmetric nor transitive.
Write $\ell \succ \ell'$ if $\ell \succeq \ell'$ and $\ell \ne \ell'$.

\begin{lemma}
\label{lemma1lorder}
 If $S$ is a rectangular domain, the following assertions hold:
\begin{enumerate}
 \item \label{item1partorder}
  Relation $\succeq$ restricted to $\bigl\{\ell \in C(S)\bigr\}$ is a partial
  order.
 \item \label{item1ellextreme}
  The elements of $C(S)$ are extremal in the following sense.
  If $\ell\in C(S)$, $\ell'\subseteq\bar S$ and $\ell\subseteq\ell'$, then
$\ell=\ell'$.
 \item \label{item1useless}
  Relations $\subseteq$ and $\succeq$ have the following concavity property.
  Suppose that $\ell\subseteq\ell_1$ and $\ell\subseteq\ell_3$ for some
$\ell\subseteq\bar S$ and $\ell_1\preceq\ell_2\preceq\ell_3\in C(S)$; then
$\ell\subseteq\ell_2$.
\end{enumerate}
\end{lemma}
It should be clear that the items above hold true.
The rigorous proof of this lemma consists of straightforward but tedious
verifications and is postponed until Appendix~\ref{sec1lorder}.

\begin{lemma}
 \label{lemma1comparable}
 Let $S$ be a rectangular domain.
If $w_\eta(\ell)>0$ and $w_\eta(\ell')>0$ for some flow field $\bar\eta$, then
$\ell$ and $\ell'$ are comparable, that is, $\ell\succeq\ell'$ or
$\ell'\succeq\ell$.
\end{lemma}
Lemma~\ref{lemma1comparable} is a consequence of the way the association rules
have been defined.
In order to prove it, one keeps applying such association
rules and the result follows by induction -- see Appendix~\ref{sec1comparable}.

The theorem below is fundamental.
It says we can decompose a given flow field in many other smaller fields, each
one corresponding to one of the maximal broken lines that cross $S$ and is
associated to the field.
In this case the original flow field and all of its features, namely the birth
process, the boundary conditions and the weight it attributes to broken traces,
are additive in the sense that each of these is obtained by summing over the
smaller fields.
On the other hand it tells that, given an ordered set of broken lines, it is
possible to combine the corresponding fields and thereby determine the flow
field that is associated to them.

\begin{theorem}
\label{theo1etaell}
Let $S$ be a rectangular domain.

Given $\ell_1\prec\cdots\prec\ell_n\in C(S)$ and $w_1,\dots,w_n>0$, there is a
unique flow field $\bar\eta$ in $\E(\bar S)$ such that
\begin{equation}
  \label{eq1lbtcarac}
  w_\eta(\ell) =
  \begin{cases} 
     w_j, & \mbox{if } \ell=\ell_j \mbox{ for some } j \\ 0, & \mbox{otherwise}
  \end{cases}
  \qquad\quad
  \mbox{for any $\ell\in C(S)$.}
\end{equation}
Moreover, for $\zeta^\circ,\xi$ such that $\bar\eta=\bar\eta(\zeta^\circ,\xi)$,
the following decompositions hold:
\begin{equation}
 \label{eq1ldecomp}
\bar\eta = \sum_jw_j\bar\eta_j, \quad \zeta^\circ = \sum_jw_j\zeta^\circ_j,
\quad \xi = \sum_jw_j\xi_j,
\end{equation}
where $\bar\eta_j=\bar\eta(\ell_j)$, $\zeta^\circ_j = \zeta^\circ(\ell_j)$ and
$\xi_j = \xi(\ell_j)$. Furthermore,
\begin{equation}
  \label{eq1lweightcalc}
  w_\eta(\ell) = \sum_{\ell'\in C(S)} w_\eta(\ell') \I_{\ell\subseteq\ell'}
 \qquad\quad
 \mbox{for any } \ell\subseteq\bar S.
\end{equation}

Conversely, given a flow field $\bar\eta$, there are unique sets $\{\ell_j\}$
and $\{w_j>0\}$ that satisfy~(\ref{eq1lbtcarac}).
The set $\{\ell_j\}$ is totally ordered by the relation `$\succeq$' and
(\ref{eq1ldecomp})-(\ref{eq1lweightcalc}) hold in this case.
\end{theorem}

\begin{proof}
We start by the converse part, first proving that (\ref{eq1lweightcalc}) holds
for any flow field $\bar\eta$, which is the most laborious work.
As we shall see, the proof of~(\ref{eq1lweightcalc}) is a formalization of the construction below,
whereas (\ref{eq1lbtcarac})~and~(\ref{eq1ldecomp}) are immediate consequences,
as discussed afterwards.

Then it will suffice to show that, given any pair of sets
$\{\ell_1\prec\ell_2\prec\cdots\prec\ell_M\}$ and $\{w_1,\dots,w_M>0\}$,
there is some flow field $\bar\eta$ satisfying~(\ref{eq1lbtcarac}).
Uniqueness of such flow field follows from the converse part.
Again, by the converse part,~(\ref{eq1ldecomp})~and~(\ref{eq1lweightcalc}) will
hold as well, completing the proof of the theorem.
In order to prove that~(\ref{eq1lbtcarac}) holds we basically have to see that
the construction below can be reversed.

Though the construction looks simple, several equivalent representations of a
flow field may be seen on the same picture.
The theorem will be deduced from these representations.
We remind that the sole condition for a field $\eta(e)\geqslant0$, $e\in\E(\bar
S)$ to be a flow field is the conservation law below:
\begin{equation}
 \label{eq1lcons2}
  \eta(e_y^\nwarrow) + \eta(e_y^\nearrow) = \eta(e_y^\swarrow) +
 \eta(e_y^\searrow)\qquad\forall\ y\in S.
\end{equation}
The construction we start describing now strongly relies on this fact.

First we plot on a horizontal line two adjacent intervals whose lengths
correspond to the flow $\eta$ on the two topmost edges of $\E(\bar S)$, i.e.,
the two edges incident from above to the topmost site $y\in S$.
Since the intervals are adjacent we may do that by marking three points on this
line.
On the next line (parallel to and below the first one), we plot four adjacent
intervals having lengths corresponding to the flow on the four edges incident
from above to the two sites on the 2nd row of $S$.
It follows from~(\ref{eq1lcons2}) that one can position these intervals in a way
that its 2nd and 4th points stay exactly below the 1st and 3rd points of the
first line.
By linking these two pairs of points we get a ``brick'' that corresponds to the
topmost site $y$ of $S$.
The width of this brick is equal to the quantity expressed in~(\ref{eq1lcons2}),
its top face is divided into two subintervals having lengths
$\eta(e_y^\nwarrow)$ and $\eta(e_y^\nearrow)$ and its bottom face is divided
into intervals of lengths $\eta(e_y^\swarrow)$ and $\eta(e_y^\searrow)$ -- see
Figure~\ref{fig1brick}.

We carry on this procedure for the 3rd horizontal line, then getting two bricks
that correspond to the two sites on the second row of $S$.
We keep doing this construction until we mark intervals corresponding to the two
bottommost edges of $\E(\bar S)$.
In the final picture we have one brick corresponding to each site of $S$.
Each pair of bricks on consecutive levels that have a common (perhaps
degenerate) interval on their boundaries corresponds to adjacent sites $y$ and
$y'$ in $S$; the length of this common interval equals $\eta\bigl(\langle
y,y'\rangle\bigr)$.
Once all the intervals have been plotted at the appropriate position, forming
all the bricks, we draw a dotted vertical line passing by each point that was
delimiting these intervals.
By doing this we divide the whole diagram into strips, completing the
construction of the \emph{brick diagram} as shown in Figure~\ref{fig1brick}.
\begin{figure}[!htb]
\tiny
\psfrag{eta(e_y^se)}{$\eta(e_y^\searrow)$}
\psfrag{eta(e_y^sw)}{$\eta(e_y^\swarrow)$}
\psfrag{eta(e_y^ne)}{$\eta(e_y^\nearrow)$}
\psfrag{eta(e_y^nw)}{$\eta(e_y^\nwarrow)$}
 \centering
 \includegraphics[width=120mm]{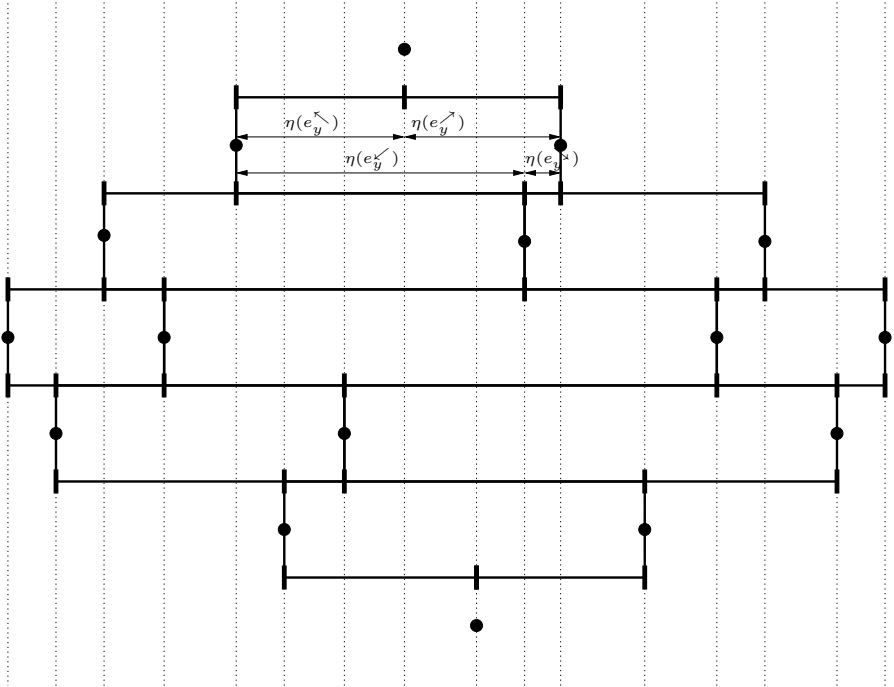}
\caption{construction of the \emph{brick diagram}.
 In this example $S$ is a $3\times3$ rectangular domain.
 There are 6 horizontal lines with a total of 24 intervals forming 9 bricks.
 The diagram is divided into 15 strips.}
\label{fig1brick}
\end{figure}

Each strip corresponds to a maximal broken line that crosses $S$ and the weight
of this broken line equals the width of the strip.
The sites/edges that compound the broken line correspond to the bricks/intervals
the strip intersects.
Given a broken trace $\ell\subseteq\bar S$, we can determine the maximal broken
line $\gamma(\ell)$ that passes through $\ell$ by looking which strips pass by
all the sites/edges of $\ell$ (i.e., their corresponding bricks/intervals); the
weight of this broken line is obtained by summing the width of such strips.
See Figure~\ref{fig1domainlines}.
\begin{figure}[!htb]
 \centering
 \includegraphics[width=75mm]{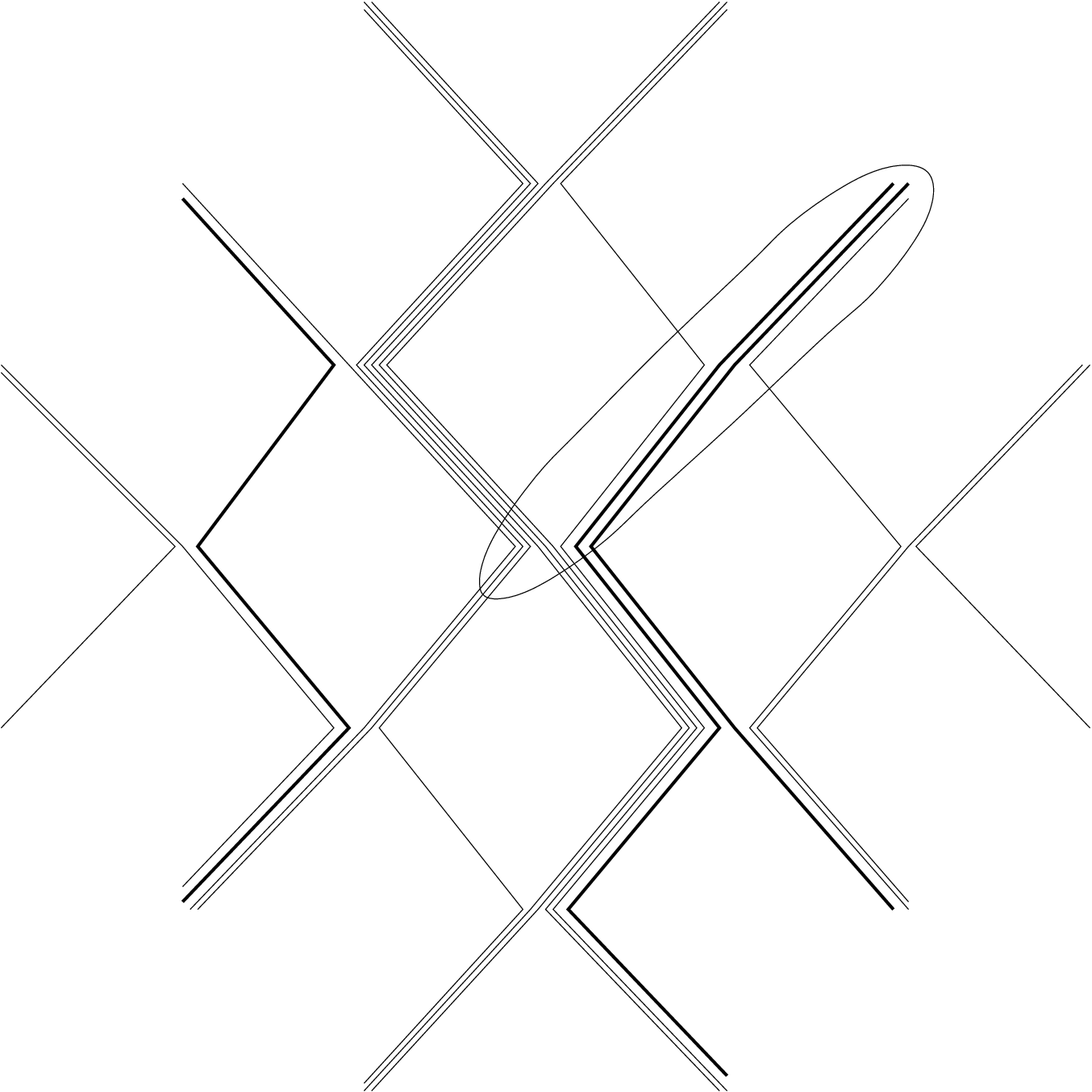}
 \hspace{0.4cm}
 \includegraphics[width=65mm]{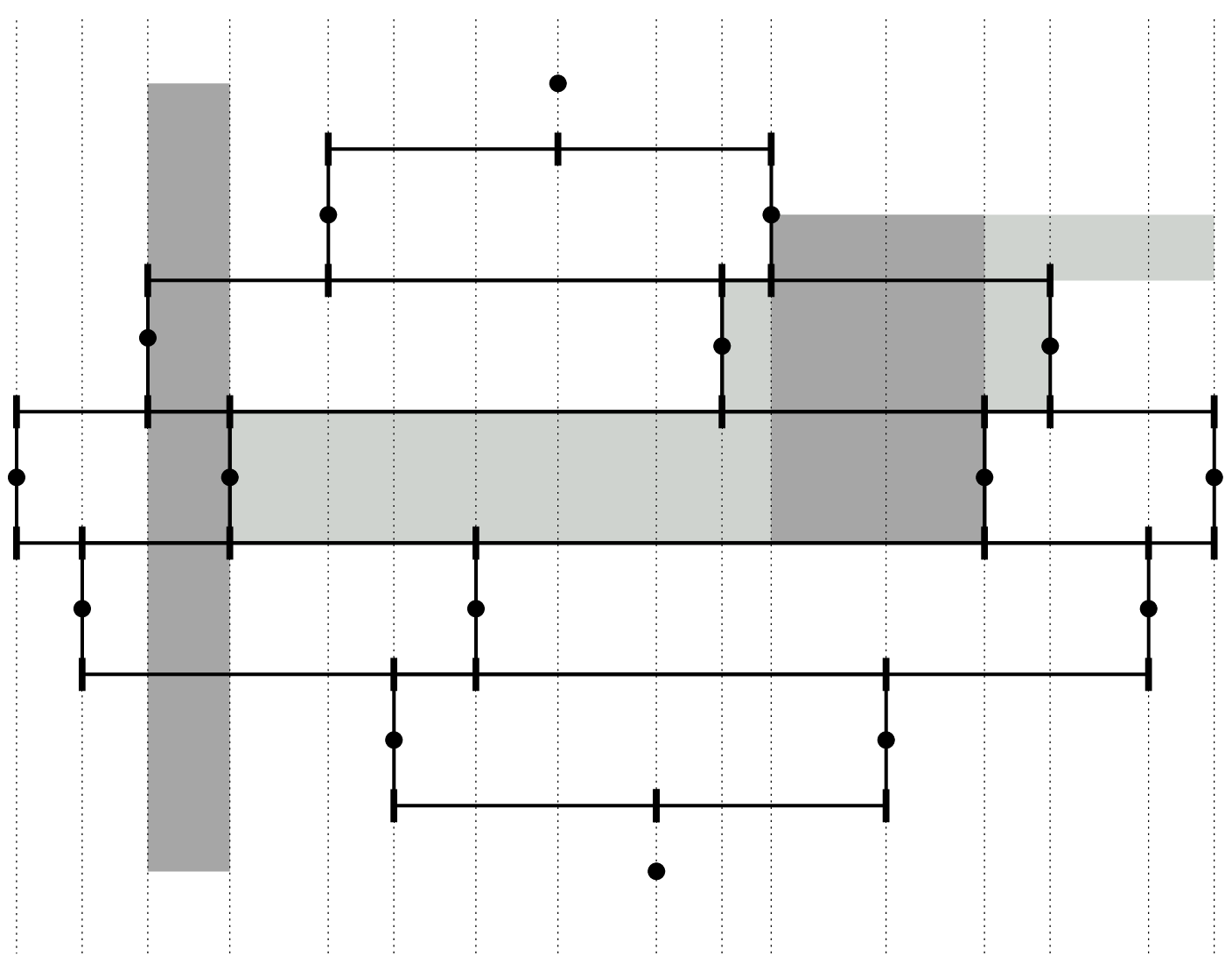}
\caption{broken lines that cross $S$ and calculation of the weight of a given
broken trace.
On the left we have a configuration of broken lines on the $3\times3$
domain $S$ and on the right we have the brick diagram that originated them.
The 3rd strip in the brick diagram is highlighted and the corresponding
broken line appears in bold on the left.
Also, for a broken trace $\ell$ that starts at the center of $S$, goes
northeast, and then exits $S$ going northeast again, we have determined which
broken lines in $C(S)$ contain $\ell$; they are the 11th and 12th ones and they
also appear in bold.
In this case $w(\ell)$ equals the sum of the widths of the 11th and 12th
strips.
There are a total of 15 broken lines crossing $S$, corresponding to the 15 strips.
}
\label{fig1domainlines}
\end{figure}

So, given a flow field $\bar\eta$ we construct the brick diagram from which the
broken line configuration is deduced, with the desired property that such broken
lines satisfy~(\ref{eq1lbtcarac}) and~(\ref{eq1lweightcalc}). 

On the other hand, let a set of well ordered broken lines, that is,
$\{\ell_1\prec\ell_2\prec\cdots\prec\ell_M\}$ and $\{w_1,\dots,w_M>0\}$, be
given.
One can consider the corresponding broken line diagram, from which one
constructs the brick diagram and the latter gives a flow field
satisfying~(\ref{eq1lbtcarac}).

In a first reading, one is encouraged to understand the above description with
the corresponding figures.
The more interested reader will find a detailed proof in
Appendix~\ref{sec1etaell}.

Since~(\ref{eq1lweightcalc}) has been proven for any flow filed, uniqueness of $\{\ell_j\}$ and $\{w_j\}$ becomes trivial by definition.
Well ordering follows from Lemma~\ref{lemma1comparable}.

Finally, (\ref{eq1ldecomp}) holds because of~(\ref{eq1lweightcalc}) when we
write each process in terms of weights of certain broken lines.
For $\zeta^\circ,\xi,\eta$ such that
$\bar\eta=\bar\eta(\zeta^\circ,\xi)=\bar\eta(\zeta^\circ ,\eta)$ we have
$\zeta^-_y = \eta(e_y^\nwarrow)$, $\zeta^+_y = \eta(e_y^\swarrow)$, $\eta^-_y =
\eta(e_y^\searrow)$, $\eta^+_y = \eta(e_y^\nearrow)$, and
$\xi_y=\eta_y^+\wedge\eta_y^-$. Given $e=\langle y,y'\rangle\in\E(\bar S)$, take
$\ell(e)=(y,e,y')$. Then of course $\eta(e)=w_\eta\bigl(\ell(e)\bigr)$.
Also, given $y\in S$, take $\ell^<(y) =(y_-,e_-,y,e_+,y_+)$, where $y=(t,x)$, $y_{\pm}=(t+1,x\pm1)$ and $e_\pm=\langle y,y_\pm\rangle$ and notice that $w_\eta\bigl(\ell^<(y)\bigr)=\xi_y$.
Also notice that $\zeta^-_y(\ell)=\I_{e_y^\nwarrow\in\ell}=
\I_{\ell(e_y^\nwarrow)\subseteq\ell}$.
Now we put all the pieces together to get $\zeta_y^-=\eta(e_y^\nwarrow) =
w_\eta\bigl(\ell(e_y^\nwarrow)\bigr)=\sum_jw_j\I_{
\ell(e_y^\nwarrow)\subseteq\ell_j}= \sum_jw_j(\zeta_j^-)_y$, i.e.,
$\zeta^-=\sum_jw_j\zeta_j^-$. The other equalities are deduced similarly.

This completes the proof of the converse part and, as discussed above, of the
theorem.
\end{proof}

\begin{corollary}
 Let $S$ be a rectangular domain and let $\bar\eta(\zeta^\circ,\xi)$ be given.
Then
\begin{equation}
 \label{eq1lhzx}
 \left[ \sum_{y\in\partial_-S} \zeta_y^+  + \sum_{y\in\partial^-S} \eta_y^-
\right] = \left[ \sum_{y\in\partial_+S} \zeta_y^- + \sum_{y\in\partial^+S}
\eta_y^+ \right] = \sum_{\ell\in C(S)}w(\ell).
\end{equation}
\end{corollary}

\begin{proof}
For each $\ell\in C(S)$ there is exactly one $e\in\ell$ such that
$y\in\partial^+S,y'\in \partial^+\bar S$ or $y\in\partial_+S,y'\in
\partial_+\bar S$, where $e=\langle y,y'\rangle$.
So
$1 =
\sum_{(t,x)\in\partial_+S} \I_{\langle(t,x),(t-1,x+1)\rangle\in\ell}
+
\sum_{(t,x)\in\partial^+S} \I_{\langle(t,x),(t+1,x+1)\rangle\in\ell}
=
\sum_{y\in\partial_+S}\zeta_y^-(\ell)
+
\sum_{y\in\partial^+S}\eta_y^+(\ell)$.
Multiplying by $w(\ell)$ and summing over all $\ell\in C(S)$ we get
$\sum_{y\in\partial_+S} \zeta_y^- + \sum_{y\in\partial^+S} \eta_y^+ =
\sum_{\ell\in C(S)}w(\ell)$.
The proof of $\sum_{y\in\partial_-S} \zeta_y^+  + \sum_{y\in\partial^-S}
\eta_y^- = \sum_{\ell\in C(S)}w(\ell)$ is similar.
\end{proof}

Denote by $H_S(\zeta^\circ,\xi)$ the quantity expressed in (\ref{eq1lhzx}) and
write $H_S(\xi)$ when $\zeta^\circ=0$.
For a birth field $\xi\geqslant0$ on $S$ we define the {directed last passage
percolation value} $G_S(\xi)$ as the maximum of $\sum_{y\in\pi} \xi_y$ over all
$\pi\in\Pi_S$, where
$$
 \Pi_S=\{\pi=(y_0,\dots,y_m):y_0\in\partial_0 S, y_m\in\partial_1 S,
t_{i+1}=t_i+1, x_{i+1}=x_i\pm1\}.
$$
When it is clear which rectangular domain we are referring to we shall drop the
subscript $S$ of $H_S$, $G_S$ and $\Pi_S$.

The next proposition illustrates the connection between broken lines and passage
time.
Furthermore, its proof gives an explicit algorithm for determining the optimal
path (which is a.s. unique when $\xi$ has continuous distribution).
\begin{proposition}
 \label{prop1llpp}
 Let $S$ be a rectangular domain.
The last passage percolation value $G_S(\xi)$ is given by the sum of the weights of the broken lines associated to the corresponding birth field: $G_S(\xi)=H_S(\xi)$.
\end{proposition}

\begin{proof}
The proof consists on formalizing the following argument: an oriented path $\pi$
connecting $\partial_0S\to\partial_1S$ can cross at most one left corner of each
broken line, and on the other hand it is possible to assemble the path
backwards, following a local rule that does not miss any broken line, this is
possible because they never cross each other. See Figure~\ref{fig1lpp}.
\begin{figure}[!htb]
\psfrag{t}{$t$}
\psfrag{x}{$x$}
 \centering
 \includegraphics[width=5.0cm]{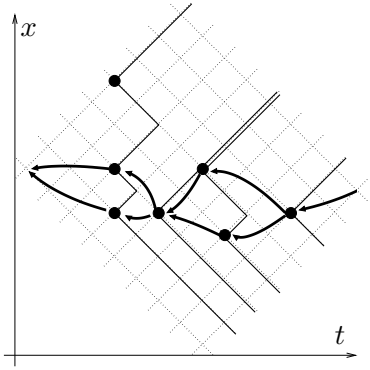}
\caption{construction of the maximal path. From $\xi$ one constructs the flow
field and by following algorithm~(\ref{eq1algorithm}) one gets an optimal path.
The theory developed for broken lines and Theorem~\ref{theo1etaell} guarantee
this is indeed optimal. The algorithm consists on forbidding the path to cross
any broken line; for that purpose it suffices to require that it does not
`cross' the flow field.}
\label{fig1lpp}
\end{figure}

Take $\bar\eta = \bar\eta(\zeta^\circ=0,\xi)$ and write $\{\ell\in
C(S):w_\eta(\ell)>0\}=\{\ell_1\prec\ell_2\prec\cdots\prec\ell_M\}$. Now by
(\ref{eq1lxidef}) and (\ref{eq1ldecomp}) it follows that $G(\xi)$ is the maximum
over all $\pi\in\Pi$ of
$$
  \sum_{y\in\pi} \xi_y = \sum_{y\in\pi} \sum_j w(\ell_j)[\xi(\ell_j)]_y = 
  \sum_j w(\ell_j) \sum_{y\in\pi} \I_{L(\ell_j)}(y).
$$
Since the paths $\pi$ are oriented, they cannot intersect more than one left
corner of each broken line, hence $\sum_{y\in\pi} \I_{L(\ell_j)}(y) \leqslant 1$
for each $j$.
We shall exhibit an algorithm for constructing a path $\pi^*$ that satisfies
\begin{equation}
 \label{eq1intersect}
 \sum_{y\in\pi^*} \I_{L(\gamma_j)}(y) \geqslant 1,
\end{equation}
which completes the proof.

The path $\pi^*$ is constructed by the following rule.
Let $y_m \in \partial_1 S$. For $i = m,m-1,m-2\dots,2,1$, let $t_{i-1}=t_i-1$
and
\begin{equation}
 \label{eq1algorithm}
  x_{i-1} = 
  \begin{cases}
     x_i-1,& \eta(e_{y_i}^\swarrow)\geqslant\eta(e_{y_i}^\nwarrow), \\ 
     x_i+1, & \mbox{otherwise}.
  \end{cases}
\end{equation}

It remains to show~(\ref{eq1intersect}),
i.e., that $\pi^*$ intersects $L(\ell_j)$ for each $j=1,\dots,M$.
Fix $j$ and write $\ell$ for $\ell_j$.
Assume without proof that $\pi^*$ intersects $\ell$; a complete proof is shown
in Appendix~\ref{sec1intersect}.
Take $n=\min\{i:y_i\in\ell\}$. If $y_0\in\ell$ we have $y_0\in L(\ell)$ due to
the fact that $\zeta^\pm_{y_0}=0$.
So suppose $n>0$. By construction $x_{n-1}=x_n\pm1$; assume for simplicity
$x_{n-1}=x_n+1$. Now $t_{n-1}=t_n-1$ and $n$ is minimal, so
$(t_n-1,x_n+1)\not\in\ell$.
Since $y_n\not\in\partial\bar S$, $x_n+1\in D(\ell)$ and thus
$(t_n+1,x_n+1)\in\ell$.
Now if $(t_n-1,x_n-1)$ were in $\ell$ there would be
$p_1\in(0,\eta(e_{y_n}^\swarrow)]$, $p_2\in(0,\eta(e_{y_n}^\nearrow)]$
associated by Case~3, which implies
$\eta(e_{y_n}^\swarrow)>\eta(e_{y_n}^\nwarrow)$, contradicting the choice of
$x_{n-1}=x_n+1$, so $(t_n-1,x_n-1)\not\in\ell$. But also $x_n-1\in D(\ell)$,
thus $(t_n+1,x_n-1)\in\ell$ and therefore $y_n\in L(\ell)$.
\end{proof}

\section{
Geometric and exponential last passage percolation}
\label{sec1appl}

% \subsection{Geometric and exponential last passage percolation}

It follows from super-additivity that the last passage percolation model satisfies a law of large numbers.
However it is interesting that for the two-dimensional model and for the special case of i.i.d.\ (geometric or exponential) passage time distributions there is an explicit expression for the limiting constant in the oriented case.
For the exponential distribution it was found by Rost~\cite{rost81} and for geometric case by Jockusch, Propp, and Shor~\cite{jockusch95}.
Large deviations were studied by Johansson~\cite{johansson00} and by Seppäläinen~\cite{seppalainen98}.
Fluctuations were studied in~\cite{johansson00}.

With the aid of the broken line theory developed in the previous sections it is possible to re-obtain the explicit constants for the law of large numbers.
We also prove exponential decay for the probability of deviations.
What we present in this section is an alternative proof that could give some
geometric insight of the model.
Besides, the broken-line approach provides an explicit, linear algorithm for
determining the maximal path (see the proof of Proposition~\ref{prop1llpp}).
In the proof we show that the boundary conditions give no asymptotic
contribution to the total flow of broken lines that cross a given domain when
they have this suitable distribution.

In the same spirit, O'Connell~\cite{oconnell00} also devises such constants by simple probabilistic arguments.
We note that our approach is self-contained, except for using of Cramér's theorem for large deviations of i.i.d. sums.
A proof of Burke's theorem is implicitly contained in our considerations of reversibility.

The construction consists on first choosing the appropriate distributions of the
boundary conditions that (i) make the broken line process reversible and (ii)
provide the correct asymptotic behavior,
and then dropping the boundary conditions afterwards.

For $N,M\in\N$, we define the last passage percolation value on the square
$\{1,\dots,N\}\times\{1,\dots,M\}$ as the random number $G(N,M)$ given by the
maximum sum of $\xi_{y_j}$ over all oriented paths $(y_1,\dots,y_{N+M-1})$ from
$(1,1)$ to $(N,M)$. $G(N,M)$ is random because so are the $\xi_y$'s.

\begin{theorem}
 \label{theo1exp}
 Suppose $\xi_y$, 
 $y\in\N^2$,
 are i.i.d and distributed as $\exp(\alpha)$ and let $\beta>0$ be fixed.
 Then a.s.
$$
 \lim_{N\to\infty} \frac1NG(N,\lfloor\beta N\rfloor) =
 \frac{\left(1+\sqrt\beta\right)^2}\alpha.
$$
For each $\delta>0$, there exists $c=c(\delta)>0$ such that
\begin{equation}
 \label{eq1concentexp}
 P\left\{\left| \frac1NG(N,\lfloor\beta N\rfloor) -
\frac{\left(1+\sqrt\beta\right)^2}\alpha \right| > \delta\right\} \leqslant
e^{-cN}
\end{equation}
for all $N\in\N$.
\end{theorem}

\begin{proof}
The central idea of the proof may be hidden among all the calculations,
basically it consists on the following argument.

By Proposition~\ref{prop1llpp} and the laws of large numbers for i.i.d. exponential r.v.'s,
$$
   G(N,\beta N) =  H_S(\xi) 
  \leqslant   H_S(\zeta^\circ,\xi) =  \sum_{y\in\partial_-S} \zeta_y^+ + 
  \sum_{y\in\partial^-S} \eta_y^- \approx \frac N {\alpha_+} + \frac {\beta N}
{\alpha_-},
$$
where $S$ is a rectangular domain with $N\times \beta N $ sites.
Here $\alpha_+$ and $\alpha_-$ can be any pair of positive numbers that make
$\alpha=\alpha_++\alpha_-$ and therefore the broken line process reversible when
the $\zeta^\pm$ are distributed as $\exp(\alpha_\pm)$ and the $\xi$ are
distributed as $\exp(\alpha)$.
As a consequence we have
\[
 \lim \frac1NG(N,\beta N) 
 \leqslant \inf_{\alpha_+,\alpha_-} \left[ \frac 1 {\alpha_+} + \frac \beta
{\alpha_-} \right ]
  = \frac{\left(1+\sqrt\beta\right)^2}\alpha,
\]
the infimum being attained for $\alpha_\pm = \alpha/(1+\beta^{\pm1/2})$.
Now we want to have the opposite inequality. We argue that
$H_S(\zeta^\circ,\xi)$ cannot be much bigger than $H_S(\xi)$. To compare both,
consider, instead of  boundary conditions $\zeta^\pm$ in $\partial S$, a
slightly enlarged domain $S'$ without boundary conditions but where to each
extra site we associate $\xi'$ corresponding to the previous $\zeta^\pm$, so
that $H_{S'}(\xi')=H_S(\zeta^\circ,\xi)$. Now for $H_{S'}(\xi')$ to be
considerably bigger than $H_S(\xi)$, it must be the case that the $\xi'$-optimal
path in $S'$ occupies a positive fraction $\epsilon$ of the boundary
$\partial_+S'$ and then takes a $\xi$-optimal path 
in the remaining $(N\times(1-\epsilon)\beta N)$-rectangle. But in fact it
happens that an oriented walker, after visiting $\epsilon\beta N$ sites in
$\partial_+ S$, looks ahead and realizes it is far too late to perform the
$\xi'$-optimal path, as the following equation shows:
\begin{equation}
\label{eq1comparetoolate}
 \epsilon\beta \frac{1+\sqrt{1/\beta}}\alpha +
\frac{(1+\sqrt{(1-\epsilon)\beta})^2}\alpha 
 = \frac{(1+\sqrt\beta)^2}{\alpha} - \frac{2\sqrt\beta}\alpha f(\epsilon).
\end{equation}
Here
\[
 f(\tau) = 1 - \tau/2 - \sqrt{1-\tau}
\]
is positive and increasing in $(0,\infty)$. Therefore the $\xi'$-optimal path in
$S'$ cannot stay too long at the boundary of $S'$ and thus 
$H_S(\xi) \approx H_{S'}(\xi')=H_{S}(\zeta^\circ,\xi)$.

Now let us move to the proper mathematical proof.

We shall use the following basic fact.
Given $\alpha_1,\alpha_2,\alpha_3,\rho,\delta,K>0$, there exist positive
constants $c,C>0$ such that if $(X_j)$, $(Y_j)$ and $(Z_j)$ are sequences of
i.i.d. r.v.'s distributed respectively as $\exp(\alpha_1)$, $\exp(\alpha_2)$ and
$\exp(\alpha_3)$, then for all $N\in\N$,
\begin{equation}
 \label{eq1llemma}
 P\left\{
 \begin{array}{r}
 \exists l,m,n\in \{0,\dots,\lceil\rho
N\rceil\}:\left|\sum_{j=1}^lX_j+\sum_{j=1}^mY_j + 
 \right.\\ \left. 
 + \sum_{j=1}^nZ_j-\frac l{\alpha_1} - \frac m{\alpha_2} - \frac n{\alpha_3}
\right|
 \geqslant \delta N - K
 \end{array}
  \right\} \leqslant Ce^{-cN},
\end{equation}
regardless of the joint distribution of $(X,Y,Z)$.

We first map our problem in $\Z^2$ to the space $\tilde\Z^2$, where the theory
of broken lines was developed. We do so by considering the rectangular domain
$S(N,M)$ given by $S(N,M)=\{(t,x)\in\tilde\Z^2:0\leqslant t+x \leqslant 2(M-1),
0\leqslant t-x \leqslant 2(N-1)\}$ and the obvious mapping between $S(N,M)$ and
$\{1,\dots,N\}\times\{1,\dots,M\}$. We write $S$ for $S(N,M)$.

In order to define a reversible broken line process in $S(N,M)$ with creation
$\xi\sim\exp(\alpha)$ we can choose $\alpha_+$ and $\alpha_-$ such that
$\alpha_++\alpha_-=\alpha$ and let $\zeta^+\sim\exp(\alpha_+)$,
$\zeta^-\sim\exp(\alpha_-)$. Take $M= M(N) = \lfloor\beta N\rfloor$. Choosing
$\alpha_+ = \alpha/(1+\beta^{1/2})$ and $\alpha_- = \alpha/(1+\beta^{-1/2})$
gives
\begin{equation}
  \label{eq1alphaopt}
  \frac 1 {\alpha_+} + \frac \beta {\alpha_-} =
\frac{\left(1+\sqrt\beta\right)^2}\alpha.
\end{equation}

By Proposition~\ref{prop1llpp},
\[
 G(N,M) = H_S(\xi) \leqslant H_S(\zeta^\circ,\xi) = \sum_{y\in\partial_-S}
\zeta_y^+
  + \sum_{y\in\partial^-S} \eta_y^-,
\]
and it follows from~(\ref{eq1llemma},\ref{eq1alphaopt}) that
\begin{equation}
 \label{eq1concentg}
 P\left\{ \frac1NG(N,M) > \frac{\left(1+\sqrt\beta\right)^2}\alpha +
\delta\right\} \leqslant C_0e^{-c_0N}.
\end{equation}

Now let us prove the lower bound to complete the concentration inequality above.
Consider 
$S'(N,M)=\{(t,x)\in\tilde\Z^2:-2\leqslant t+x \leqslant 2(M-1), -2\leqslant t-x
\leqslant 2(N-1)\}$ and for $0\leqslant n\leqslant\beta N$, take \\
$\tilde S(N,M-n)=\{(t,x)\in\tilde\Z^2: 2n\leqslant t+x \leqslant 2(M-1),
0\leqslant t-x \leqslant 2(N-1)\}$, \\
$\tilde S'(N,M-n)=\{(t,x)\in\tilde\Z^2:
2n-2\leqslant t+x \leqslant 2(M-1), -2\leqslant t-x \leqslant 2(N-1)\}$, \\
$\tilde S(N-n,M)=\{(t,x)\in\tilde\Z^2: 0\leqslant t+x \leqslant 2(M-1),
2n\leqslant t-x \leqslant 2(N-1)\}$, \\ 
$\tilde S'(N-n,M)=\{(t,x)\in\tilde\Z^2: -2\leqslant t+x \leqslant 2(M-1),
2n-2\leqslant t-x \leqslant 2(N-1)\}$.

For $(\zeta^\circ,\xi)$ defined on $S$ take $\xi'$ on $S'$ given by
\[
  \xi'_{t,x} = 
  \begin{cases}
     \xi_{t,x}, & y\in S \\
     \zeta^-_{t+1,x-1}, & t+1,x-1\in \partial_+S \\ 
     \zeta^+_{t+1,x+1}, & t+1,x+1\in \partial_-S \\ 
     0 & \mbox{otherwise}.
  \end{cases}
\]
For $(\tilde\zeta^\circ,\xi)$ defined on $\tilde S$ take $\tilde\xi'$ on $\tilde S'$ given by the analogous formulae and notice that $H_S(\zeta^\circ,\xi)=H_{S'}(\xi')$ and $H_{\tilde S}(\tilde\zeta^\circ,\xi)=H_{\tilde S'}(\tilde\xi')$.

The two facts below will be important:
\begin{eqnarray}
 \label{eq1somenopt}
 \mbox{For some } n ,&&
 H_{S'}(\xi') = \sum_{j=0}^{n}\xi'_{j-1,j+1} + H_{\tilde S(N,M-n)}(\xi) \mbox{
or }\\&&\nonumber
 H_{S'}(\xi') = \sum_{j=0}^{n}\xi'_{j-1,-j-1} + H_{\tilde S(N-n,M)}(\xi).  \\
 \label{eq1allnopt}
 \mbox{For all } n ,&& H_{S}(\xi) \geqslant   H_{\tilde S(N,M-n)}(\xi) \mbox{
and } \\&&\nonumber
 H_{S}(\xi) \geqslant   H_{\tilde S(N-n,M)}(\xi).
\end{eqnarray}
Given any $0<\delta'<\delta$, by putting (\ref{eq1somenopt}) and
(\ref{eq1allnopt}) together we see that for
\begin{equation}
 \label{eq1smalllpp}
 G(N,M) \leqslant N\left[ \frac{\left(1+\sqrt\beta\right)^2}\alpha - \delta
\right]
\end{equation}
to hold, we must have either
\begin{equation}
 \label{eq1ldpreversiblebl}
 H_{S'}(\xi') \leqslant N\left[ \frac{\left(1+\sqrt\beta\right)^2}\alpha -
\delta' \right],
\end{equation}
or, for some $n\in\{1,\dots,M-n\}$,
\begin{eqnarray}
 & \left( \sum_{j=0}^n \zeta_{j,j}^- \right ) + H_{\tilde S(N,M-n)}(\tilde\xi)
\geqslant 
 N\left[ \frac{\left(1+\sqrt\beta\right)^2}\alpha - \delta' \right]
 \label{eq1largewayaround} \\
 & \sum_{j=0}^n \zeta_{j,j}^- \geqslant N[\delta-\delta']
 \label{eq1largerecovering},
\end{eqnarray}
or, for some $n\in\{1,\dots,N-n\}$,
\begin{eqnarray}
 & \left( \sum_{j=0}^n \zeta_{j,-j}^+ \right ) + H_{\tilde S(N-n,M)}(\tilde\xi)
\geqslant 
 N\left[ \frac{\left(1+\sqrt\beta\right)^2}\alpha - \delta' \right]
 \label{eq1largewayaround2} \\
 & \sum_{j=0}^n \zeta_{j,-j}^+ \geqslant N[\delta-\delta']
 \label{eq1largerecovering2}.
\end{eqnarray}

The probability of (\ref{eq1ldpreversiblebl}) decays exponentially fast and
this can be shown exactly as was done for (\ref{eq1concentg}).

We consider now the other possibility,
(\ref{eq1largewayaround},\ref{eq1largerecovering}).
The case (\ref{eq1largewayaround2},\ref{eq1largerecovering2}) is treated in a
completely analogous way and is thus omitted. Let $\delta>0$ be fixed, take
$\epsilon_0 = {\delta\alpha}/[{2\beta(1+\sqrt{1/\beta})}]$ and $\delta' =
\frac\delta3 \wedge [\frac{\sqrt{\beta}}{\alpha} f(\epsilon_0)]$. With this
choice of parameters 
\begin{equation}
 \label{eq1epssmallrecov}
 \epsilon\beta\frac{1+\sqrt{1/\beta}}\alpha \leqslant \frac
 \delta2 < \delta-\delta'
\end{equation} 
holds for $\epsilon\leqslant\epsilon_0$ and
\begin{eqnarray}
 \begin{array}{c}
 \label{eq1longsum}
 \displaystyle  \epsilon\beta\frac{1+\sqrt{1/\beta}}\alpha + 
 (\beta-\epsilon\beta)\frac{1+(\beta-\epsilon_0\beta)^{-1/2}}\alpha +
\frac{1+\sqrt{\beta-\epsilon_0\beta}}\alpha 
  \leqslant \\ \qquad \qquad \qquad \qquad \qquad \qquad 
  \displaystyle \leqslant \frac{(1+\sqrt\beta)^2}\alpha - 2 \delta'
 \end{array}
\end{eqnarray}
holds for $\epsilon\geqslant\epsilon_0$.

Consider the event that (\ref{eq1largerecovering}) happens for some $0\leqslant
n\leqslant \lceil \epsilon_0\beta N \rceil = M_0$.
Since $\zeta_j^- =_d \exp\big(\alpha/(1+\sqrt{1/\beta})\big)$, it follows from~(\ref{eq1epssmallrecov}) and~(\ref{eq1llemma}) that the probability of this event decays exponentially fast in $N$.

It remains to consider $n \geqslant M_0$ and show that in this case it is the
probability of~(\ref{eq1largewayaround}) that decays exponentially fast. Now
\begin{eqnarray*}
&&
P\left\{
 \exists n \in\{M_0,\dots,M\}: 
 \sum_{j=0}^n \zeta_{j,j}^- + H_{\tilde S}(\xi) \geqslant 
 N \frac{\left(1+\sqrt\beta\right)^2}\alpha - N\delta'
\right\}
\leqslant \\&&
P\left\{
 \exists n \in\{M_0,\dots,M\}: 
 \sum_{j=0}^n \zeta_{j,j}^- + H_{\tilde S'}(\tilde\xi') \geqslant 
 N \frac{\left(1+\sqrt\beta\right)^2}\alpha - N\delta'
\right\}
= \\&&
P\left\{
 \exists n : 
 \sum_{j=0}^n \zeta_{j,j}^-  + \sum_{j=n}^{M-1} \tilde\zeta_{j,j}^- 
 + \sum_{j=0}^{N-1} \tilde\eta^+_{M-1+j,M-1-j} \geqslant 
 N \frac{\left(1+\sqrt\beta\right)^2}\alpha - N\delta'
\right\},
\end{eqnarray*}
where $\tilde\zeta^\pm$ are distributed as
$\exp(\alpha/[1+(\beta-\epsilon_0\beta)^ {\pm1/2}])$ so that the broken line
process on $\tilde S$ with $\xi$ distributed as $\exp(\alpha)$ is reversible,
and therefore the $\tilde\eta^\pm$ are also distributed as the
$\tilde\zeta^\pm$.

By~(\ref{eq1longsum}) the right-hand side of the inequality in the last line is
greater than
\[
  n \frac{1+\sqrt{1/\beta}}\alpha + 
  (M-n)\frac{1+(\beta-\epsilon_0\beta)^{-1/2}}\alpha +
  N\frac{1+\sqrt{\beta-\epsilon_0\beta}}\alpha + N \delta'
\]
and by~(\ref{eq1llemma}) the last probability above decays exponentially fast.
The proof is finished.
\end{proof}

\begin{theorem}
 \label{theo1llngeo}
 Suppose $\xi_y$,  
 $y\in\N^2$,
 are i.i.d and distributed as ${\rm Geom}(\lambda)$, $\lambda\in(0,1)$ and let
$\beta>0$ be fixed. Then a.s.
$$
  \lim_{N\to\infty} \frac1NG(N,\lfloor\beta N\rfloor) = \frac{
\left(1+\sqrt{\beta \lambda}\right)^2} {1-\lambda} -1.
$$
For each $\delta>0$, there exists $c=c(\delta)>0$ such that
\[
 P\left\{\left| \frac1NG(N,\lfloor\beta N\rfloor) - \left( \frac{
\left(1+\sqrt{\beta \lambda}\right)^2} {1-\lambda} -1\right) \right| >
\delta\right\} \leqslant e^{-cN}
\]
for all $N\in\N$.
\end{theorem}

\begin{proof}
The proof is absolutely identical to that of the previous theorem, so we just
highlight which equations should be replaced by their analogous.

In the heuristic part take $\lambda_+=[(\lambda+\sqrt{\beta
\lambda})/(1+\sqrt{\beta \lambda})]\in(\lambda,1)$ and
$\lambda_-=\lambda/\lambda_+\in(\lambda,1)$, so that
$\lambda=\lambda_+\lambda_-$, the process is reversible for
$\zeta^\pm\sim\exp(\lambda_\pm)$ and 
$$
  \frac{\lambda_+}{1-\lambda_+}+\beta\frac{\lambda_-}{1-\lambda_-} = 
  \frac{ \left(1+\sqrt{\beta \lambda}\right)^2} {1-\lambda} -1.
$$
Instead of (\ref{eq1comparetoolate}) consider
\[
 \epsilon\beta\frac{\lambda_-}{1-\lambda_-} +
\frac{(1+\sqrt{(1-\epsilon)\beta\lambda})^2}{1-\lambda}-1
 = \frac{(1+\sqrt{\beta \lambda})^2}{1-\lambda}-1 - \frac{2\sqrt{\beta
\lambda}}{1-\lambda}f(\epsilon).
\]

The proof of
\[
 P\left\{ \frac1NG(N,\lfloor\beta N\rfloor) > \frac{ \left(1+\sqrt{\beta
\lambda}\right)^2} {1-\lambda} -1 + \delta\right\} \leqslant e^{-cN}
\]
is analogous to the proof of (\ref{eq1concentg}).

For the opposite inequality, we take
$\epsilon_0=\delta(1-\lambda)/[2(\lambda\beta+\sqrt{\lambda\beta})]$, $\delta' =
\frac\delta3 \wedge [\frac{\sqrt{\beta\lambda}}{1-\lambda} f(\epsilon_0)]$ so
that instead of (\ref{eq1epssmallrecov}) and (\ref{eq1longsum}) the following
estimates will hold, respectively for $\epsilon\leqslant\epsilon_0$ and
$\epsilon\geqslant\epsilon_0$:
\[
 \epsilon\beta \frac{\beta\lambda+\sqrt{\beta\lambda}}{\beta(1-\lambda)} 
 \leqslant \frac \delta2 < \delta-\delta',
\]
\begin{eqnarray*}
 \epsilon\beta \frac{\beta\lambda+\sqrt{\beta\lambda}}{\beta(1-\lambda)} +
 (\beta-\epsilon\beta)\frac{(\beta-\epsilon_0\beta)
\lambda+\sqrt{(\beta-\epsilon_0\beta)\lambda}}{
(\beta-\epsilon_0\beta)(1-\lambda)}  +
 \frac{\lambda+\sqrt{(\beta-\epsilon_0\beta)\lambda}}{1-\lambda}
 \leqslant \\ \leqslant \frac{(1+\sqrt{\beta\lambda})^2}{1-\lambda} - 1 - 2
\delta'.
\end{eqnarray*}
The rest of the proof is the same.
\end{proof}

\appendix

\section{Existence of the maximal broken line}
\label{sec1maximal}

Here we prove that given a flow field and a broken trace there is a maximal
broken line associated to that field and having that trace.

We claim that there exist unique $J_1=(a_1,b_1], J_2=(a_2,b_2],
\dots,J_n=(a_n,b_n]$ such that 
\begin{equation}
 \label{eq1gammaell}
 (e_1,J_1)\sim(e_2,J_2)\sim\cdots\sim(e_n,J_n),
\end{equation}
and such that any $p_1,p_2,\dots,p_n$ with property
\begin{equation}
 \label{eq1pgammaell}
 (e_1,p_1)\sim(e_2,p_2)\sim\cdots\sim(e_n,p_n)
\end{equation} must satisfy $p_i\in J_i,\ i=1,\dots,n$.
If $J_1\ne\emptyset$ we define $\gamma(\ell) =
(y_0,e_1,J_1,y_1,\dots,e_n,J_n,y_n)$ and by~(\ref{eq1gammaell}) we have
$\gamma(\ell) \in B(\bar\eta)$; otherwise $\gamma(\ell)=\emptyset$.

To prove the claim start by observing some consequences of the association
rules.
\\ 1. If $(e_1,p_1)\sim(e_2,p_2)$ then $(e_1,p_1-\delta)\sim(e_2,p_2-\delta)$
for some $\delta>0$.
\\ 2. If $(e_1,p_1^k)\sim(e_2,p_2^k)\ \forall\ k\in\N$ with $p_1^k\uparrow p_1$
then $p_2^k\uparrow p_2$ and $(e_1,p_1)\sim(e_2,p_2)$.
\\ 3. If $(e_1,p_1)\sim(e_2,p_2)$ and $(e_1,p_1')\sim(e_2,p_2')$ with $p_1'>p_1$
then $p_1'-p_2'=p_1-p_2$ and $(e_1,(p_1,p_1'])\sim(e_2,(p_2,p_2'])$.

Now let $A$ be the set of $p_1\in(0,\eta(e_1)]$ for which it is possible to find
$p_2,p_3,\dots,p_n$ such that~(\ref{eq1pgammaell}) holds. Suppose that
$A\ne\emptyset$ and take $a_1=\inf A$, $b_1=\sup A$. (When $A=\emptyset$ we take
$J_i=\emptyset$ and the desired properties hold trivially.) Consider a sequence
$(p_1^k)\subseteq A$ with $p_1^k\uparrow b_1$. By Property~2 above we have
$(e_1,b_1)\sim\cdots\sim(e_n,b_n)$ and $b_1\in A$. It follows from Property~1
that $a_1<b_1$ and we can take another sequence $(p_1^k)\subseteq A$ with
$b_1>p_1^k\downarrow a_1$, $(e_1,p_1^k)\sim\cdots\sim(e_n,p_n^k)$. By property~3
$(e_1,(p_1^k,b_1])\sim\cdots\sim(e_n,(p_n^k,b_n])$ and $b_i>p_i^k\downarrow a_i$
for $i=1,\dots,n$; thus $(e_1,(a_1,b_1])\sim\cdots\sim(e_n,(a_n,b_n])$ and
$b_1-a_1=\cdots=b_n-a_n$. If it were the case that $a_1>0$ and $a_1\in A$, by
Property~1 it would hold that $a_1-\delta\in A$ contradicting $a_1=\inf A$.
Therefore $A=(a_1,b_1]$. Suppose $p_1,\dots,p_n$ satisfy (\ref{eq1pgammaell});
by definition $p_1\in A=(a_1,b_1]$ and by Property~3 we have $b_i-p_i=b_1-p_1\in
[0,b_1-a_1)=[0,b_i-a_i)$, that is, $p_i\in(a_i,b_i]$. As a consequence we have
that $J_i'\subseteq(a_i,b_i], i=1,\dots,n$ whenever
$(e_1,J_1')\sim\cdots\sim(e_n,J_n')$, from which uniqueness follows.

\section{Proof of Lemma~\ref{lemma1lorder}}
\label{sec1lorder}

We start proving Item~\ref{item1partorder}.
Relation $\succeq$ is obviously reflexive.
We now show that it is antisymmetric. Let $\ell\succeq\ell'\succeq\ell$, we want
to prove that $\ell=\ell'$. First suppose $D(\ell)\cap D(\ell')\ne\emptyset$ and
take $x_0\in D(\ell)\cap D(\ell')$. Write
$D(\ell)=\{x_{-n},\dots,x_0,\dots,x_m\}$,
$D(\ell')=\{x_{-n'},\dots,x_0,\dots,x_{m'}\}$ and $D(\ell)\cap
D(\ell')=\{x_{-\tilde n},\dots,x_0,\dots,x_{\tilde m}\}$ with $x_{i+1}=x_i+1$
and $\tilde m = m\wedge m', \tilde n = n\wedge n'$.
Since $\ell\succeq\ell'\succeq\ell$ we have $t(x_i)=t'(x_i)$ for $i=-\tilde
n,\dots,\tilde m$. All we need to show is that $n=n'$ and $m=m'$.
Suppose $\tilde m=m'$. Then $(t'(x_{\tilde m}),x_{\tilde m})\in\partial_+\bar
S\cup\partial^+\bar S$; and since $t(x_{\tilde m})=t'(x_{\tilde m})$ we cannot
have $m>\tilde m$, thus $m=\tilde m$. By the same argument, if $\tilde m=m$ we
conclude $m'=\tilde m$, therefore $m=\tilde m=m'$. Similarly we show that
$n=n'=\tilde n$.
It remains to consider the case $D(\ell)\cap D(\ell')\ne\emptyset$, which is
ruled out by the following claim.
\begin{claim}
 \label{claim1empty}
 If $\ell,\ell'\in C(S)$ and $D(\ell)\cap D(\ell')=\emptyset$ then $I(\ell)\cap
 I(\ell')=\emptyset$.
\end{claim}
So let us prove the claim. We'll show that $I(\ell)\cap I(\ell')\ne\emptyset$
implies $D(\ell)\cap D(\ell')\ne\emptyset$. Let $\tilde t\in I(\ell)\cap
I(\ell')$ and take $x''\in D(\ell), x'\in D(\ell')$ such that
$t(x'')=t'(x')=\tilde t$. Assume for simplicity $x''<x'$ and let $(\bar t,\bar
x)$ denote the topmost site of $S$. Since $(\tilde t,x')\in\bar S$ we have $\bar
x-\bar t\geqslant x'-\tilde t-2$ and $\bar x +\bar t\geqslant x'+\tilde t-2$.
Now $\ell$, after passing by $(\tilde t,x'')$ when going upwards, must cross
either of the lines $\{(x,t):x-t=\bar x-\bar t\}$ or $\{(x,t):x+t=\bar x+\bar
t\}$, because $\ell\in C(S)$. After crossing either of these lines there will be
$(t_*,x_*)\in\ell$ with $(x_*-1)-(t_*+1)=\bar x - \bar t$ or
$(x_*-1)+(t_*-1)=\bar x + \bar t$, respectively. Therefore $x^*\geqslant
x'-|t'-t_*|$. But by~(\ref{eq1lbroktr})  we have $|t'-t_*|\leqslant|x_*-x''|$,
so $x'-x_*\leqslant |x_*-x''|$. Assuming $x_*\leqslant x'$ (the other
possibility trivially implies the desired result), one has
$|x'-x_*|\leqslant|x''-x_*|$ thus $x_*\geqslant \frac{x'+x''}2$ and therefore
$\frac{x'+x''}2\in D(\ell)$. Analogously we show that $\frac{x'+x''}2\in
D(\ell')$ and the proof is done.

Finally let us see that $\succeq$ is transitive. For a given point
$y_*=(t_*,x_*)\in\bar S$, define $A(t_*,x_*)=\{(t,x)\in \tilde\Z^2: t-x\geqslant
t_*-x_*,t+x\geqslant t_*+x_*\}=\{(t,x)\in \tilde\Z^2: t\geqslant t_*+|x-x_*|\}$ 
and for $\ell\subseteq\bar S$, define $A(\ell)=\cup_{y\in\ell}A(y)$.
Notice that $A(y)\subseteq A(\tilde y)$ iff $y\in A(\tilde y)$, so $A(\ell')\subseteq A(\ell)$ is equivalent to $\ell'\subseteq A(\ell)$. Now let $\ell\succeq\ell'\succeq\ell''$.
It follows from these observations and from the
claim below that $\ell\succeq\ell''$.
\begin{claim}
 Let $\ell,\ell'\in C(S)$.
 Then $\ell'\succeq\ell$ if and only if $\ell'\subseteq
A(\ell)$.
\end{claim}
We start proving the `only if' part of the claim. Let $\ell'\succeq\ell\in
C(S)$. Suppose $D(\ell)\cap D(\ell')=\emptyset$. By Claim~\ref{claim1empty}
$I(\ell)\cap I(\ell')=\emptyset$ and, since $t'(x')>t(x)$ for some $x'\in
D(\ell')$, $x\in D(\ell)$, we have $t'(x')>t(x)$ for all $x'\in D(\ell')$, $x\in
D(\ell)$. Consider the case $x>x'\ \forall\ x'\in D(\ell')$, $x\in D(\ell)$; the
other situation is analogous.
Writing $D(\ell)=\{x_0<\cdots<x_n\}$ and $t_0=t(x_0)$, we must have
$(t_0,x_0)\in\partial_-\bar S$ or $(t_0,x_0)\in\partial^-\bar S$ and the latter
is ruled out since there is $x'\in D(\ell')$ with $x'<x,t'(x')>t_0$ in $\bar S$.
Now as $(t_0,x_0)\in\partial_-\bar S$ we have $t+x\geqslant t_0+x_0\ \forall\
(t,x)\in\bar S$ and, as $x'<x_0,t'(x')>t_0\ \forall\ x'\in D(\ell')$ we have
$t'-x'>t_0-x_0$ for all $(t',x')\in\ell'$. Therefore $\ell'\subseteq
A(t_0,x_0)\subseteq A(\ell)$.
Suppose on the other hand that $D(\ell)\cap D(\ell')\ne\emptyset$,
take $x_0\in D(\ell)\cap D(\ell')$ and write
$D(\ell)=\{x_{-n},\dots,x_0,\dots,x_m\}$,
$D(\ell')=\{x_{-n'},\dots,x_0,\dots,x_{m'}\}$ and $D(\ell)\cap
D(\ell')=\{x_{-\tilde n},\dots,x_0,\dots,x_{\tilde m}\}$ with $x_{i+1}=x_i+1$
and $\tilde m = m\wedge m', \tilde n = n\wedge n'$. For $x\in D(\ell)\cap
D(\ell')$ we have $t'(x)\geqslant t(x)$ and of course $(t'(x),x)\in A(t(x),x)$.
Take $t_m=t(x_m)$, by definition $(t_m,x_m)\in\partial_+\bar S$ or
$(t_m,x_m)\in\partial^+\bar S$.
In the latter case it must be that $m\geqslant m'$, for if we suppose that
$m'\geqslant m$, then as $t'(x_m)\geqslant t(x_m)$, we must have
$t'(x_m)=t(x_m)$, thus $(t'(x_m),x_m)\in\partial^+\bar S$ and $m'=m$. In the
former case we have $\{(t,x)\in\bar S:x\geqslant x_m\}\subseteq A(t_m,x_m)$.
Therefore $(t'(x),x)\in A(\ell)$ for $x\in\{x_0,\dots,x_{m'}\}$. Analogous
arguments show that $(t'(x),x)\in A(\ell)$ for $x\in\{x_{-n'},\dots,x_0\}$.
The `if' part is shorter. Suppose $\ell'\subseteq A(\ell)$. Take some $x'\in
D(\ell')$, there is $(t,x)\in\ell$ such that $(t'(x'),x')\in A(t,x)$ and thus
$t'(x')\geqslant t+|x-x'|\geqslant t$. Now for $x'\in D(\ell')\cap D(\ell)$,
there is some $\tilde x\in D(\ell)$ such that $(t'(x'),x')\in A(t(\tilde
x),\tilde x)$. We want to show that $t'(x')\geqslant t(x')$ and it follows from
the fact that $\ell\subseteq A^-(t,x)$ for any $(t,x)\in\ell$, where
$A^-(t_*,x_*)=\{(t,x)\in \tilde\Z^2: t\geqslant t_*+|x-x_*|\}$.

Item~\ref{item1ellextreme} is easy.
If $\ell=(y_a,y_{a+1},\dots,y_b)\subseteq\ell'=(y_{a'},y_{a'+1},\dots,y_b')$
with $a'\leqslant a<b\leqslant b'$ for $\ell\in C(S)$, $\ell'\subseteq\bar S$,
then $a=a'$, $b=b'$ and therefore $\ell=\ell'$; for if $a'<a$ we would have
$y_a\in S$, contradicting $\ell\in C(S)$, same for $b>b'$.
Item~\ref{item1useless} is not used in this work and we omit its proof.

\section{Proof of Lemma~\ref{lemma1comparable}}
\label{sec1comparable}
Let $\ell,\ell'\subseteq\bar S$ such that $w(\ell),w(\ell')>0$.
If $t(x)=t'(x)\ \forall x\in D(\ell)\cap D(\ell')$ the result is trivial.
So suppose there is $x_0$ such that $t(x_0)\ne t'(x_0)$ and assume for
simplicity that $t(x_0)<t'(x_0)$. Write $D(\ell)\cap D(\ell') = \{x_{-\tilde
m},\dots,x_{-1},x_0,x_1,x_{\tilde n}\}$,
$\gamma(\ell)=\{y_{-m},e_{-m+1},J_{-m+1},y_{-m+1},\dots,e_0,J_0,y_0,\dots,e_m,
J_m,y_m\}$ and
\\
$\gamma(\ell')=\{y'_{-m'},e'_{-m'+1},J'_{-m'+1},y'_{-m'+1},\dots,e'_0,J'_0,y'_0,
\dots,e'_{m'},J'_{m'},y'_{m'}\}$. We want to show that $t(x_i)\leqslant t'(x_i)$
for $i=0,1,\dots,\tilde n$, the proof for $i=0,-1,\dots,-\tilde m$ is analogous.

We claim that, for $i=0,1,\dots,\tilde n-1$, the following facts hold:
$t(x_i)\leqslant t'(x_i)$, $t(x_{i+1})\leqslant t'(x_{i+1})$ and, in case
$t(x_i) = t'(x_i)$, $t(x_{i+1}) = t'(x_{i+1})$ we have $J_{i+1}\prec J'_{i+1}$,
i.e., $p<p'$ for all $p\in J_{i+1}$, $p'\in J'_{i+1}$. Let us prove the claim by
induction.
For $i=0$ the result is obvious because of~(\ref{eq1lbroktr}) and
$t(x_0)<t'(x_0)$.
Suppose the claim is true for $i=k-1$. We have three possibilities.
Case~A: $t(x_k)=t'(x_k)$ and $t(x_{k-1})=t'(x_{k-1})$.
Case~B: $t(x_k)=t'(x_k)$ and $t(x_{k-1})<t'(x_{k-1})$.
Case~C: $t(x_k)<t'(x_k)$.
In Case~C the claim holds for $i=k$ for the same reason it holds for $i=0$. In
Case~A $J_k\prec J'_k$ and, because of the association rules, we have
$t(x_{k+1})<t'(x_{k+1})$ or $t(x_{k+1})=t'(x_{k+1})$ with $J_{k+1}\prec
J'_{k+1}$; either way the claim holds for $i=k$.
In case~B, since $t(x_{k-1})<t'(x_{k-1})$, by the association rules we cannot
have $t(x_{k+1})>t'(x_{k+1})$ and if $t(x_{k+1})=t'(x_{k+1})$ we must have
$J_{k+1}\prec J'_{k+1}$, so the claim holds.

The proof is complete.

\section{Formalization of the brick diagram}
\label{sec1etaell}
In this appendix we formalize the construction of the brick diagram described in
the proof of Theorem~\ref{theo1etaell}. Then we show that~(\ref{eq1lweightcalc})
holds for any flow field in the converse part of the theorem and
that~(\ref{eq1lbtcarac}) holds for some flow field in its direct part.

As mentioned in the proof of the theorem, there are several equivalent
representations of a flow field and this proof relies on them. We describe how
to obtain from each representation the next one, and we mention the properties
of each representation that guarantee it is possible to come back to the
previous setting. Statements will be made without proof when their verification
is a bare tedious routine.
To simplify the presentation, assume that $S$ has the form
$S=\{(t,x)\in\tilde\Z^2:|t|+|x|\leqslant N\}$ for some $N\in2 \Z_+$.

Let $S^* = \{y\in\Z^2:\exists y'\in S:|y-y'|=1\}$. Notice that
$S^*\subseteq(\tilde\Z^2)^*=\Z^2\backslash\tilde\Z^2$.
The first alternative representation of a flow field is $p_{t,x}, (t,x)\in S^*$,
defined below.

Take $W_{-N-1}=0$, for $i=-N,\dots,0$ take $W_i = W_{i-1} +
\eta(e_{i,N+i}^\nwarrow)$ and for $i=1,\dots,N+1$ take $W_i = W_{i-1} +
\eta(e_{i-1,N-i+1}^\nearrow)$.
Let $p_{0,N+1} = W_0$.

For each $x=N,N-1,\dots,1,0,-1,\dots,-N,-N-1$, set
\[
 p_{t,x} = p_{t+1,x+1} - \eta(e_{t+1,x}^\nwarrow) 
\]
for $t=-(N+1-x),-(N+1-x)+2,\dots,(N+1-x)-2$ and
\[
 p_{t,x} = p_{t-1,x+1} + \eta(e_{t-1,x}^\nearrow)
\]
for $t=N+1-x$.

With successive uses of (\ref{eq1lcons2}) it is not hard to see that $p_{t,x}$
has the following properties:
\begin{eqnarray*}
  && p_{-N-1,0}=0,
  \\ &&
  p_{t,x} \leqslant p_{t+1,x\pm1}\qquad \mbox{whenever they belong to }S^*
\end{eqnarray*}
and that it is possible to re-obtain $\bar\eta$ from the $p(t,x)$ (writing $y=(t,x)$) by
\begin{equation}
\label{eq1etap}
\begin{array}{l}
  \eta(e_y^\swarrow) = p_{t,x-1} - p_{t-1,x}
\\
  \eta(e_y^\searrow) = p_{t+1,x} - p_{t,x-1}
\\
  \eta(e_y^\nwarrow) = p_{t,x+1} - p_{t-1,x}
\\
  \eta(e_y^\nearrow) = p_{t+1,x} - p_{t,x+1}.
\end{array}
\end{equation}

Now consider the set of all $p_{t,x}$ and reorder it by taking
$A=\{q_0,\dots,q_M\}=\{p_{t,x}: (t,x)\in S^*\}$ with $q_0<q_1<\cdots<q_M$. Then
$q_0=W_{-N-1}=p_{-N-1,0}=0$ and $q_M=W_{N+1}=p_{N+1,0}$.

For $(t,x)\in S^*$, take $k(t,x)\in\{0,\dots,M\}$ as the unique subindex $k$
that satisfies $p_{t,x}=q_{k}$.
Then
\begin{eqnarray}
  && k(-N-1,0)=0 
  \nonumber \\ &&
  k(N+1,0)=M
  \nonumber \\ &&
  k(t,x) \leqslant k(t+1,x\pm1)\qquad \mbox{whenever they belong to }S^*
  \label{eq1kconsist} \\ &&
 \mbox{For all } k=0,\dots,M \mbox{ there is } y\in S^* \mbox{ such that }
k(y)=k.
  \nonumber
\end{eqnarray}
Of course it is possible to re-obtain $p_{t,x}$ from $A$ and $k(t,x)$:
\[
  p_{t,x} = q_{k(t,x)}.
\]

For $y=(t,x)\in \bar S$, take $k_\pm(y)\in\{0,\dots,M\}$ as
\[
 k_-(t,x) =
 \begin{cases}
   k(t-1,x),& y\in S\cup\partial^+\bar S\cup\partial^-\bar S \\ 
   k(t,x-1),& y\in\partial_+(S) \\
   k(t,x+1),& y\in\partial_-(S).
 \end{cases}
\]
\[
 k_-(t,x) = 
 \begin{cases}
   k(t+1,x),& y\in S\cup\partial_+\bar S\cup\partial_-\bar S \\ 
   k(t,x-1),& y\in\partial^+(S) \\
   k(t,x+1),& y\in\partial^-(S).
 \end{cases}
\]
Then
\begin{eqnarray*}
  && k_-(-N,0)=0
  \\ &&
  k_+(N,0) = M
  \\ &&
  k_+(t,x) \leqslant k_+(t+1,x\pm1)\qquad \mbox{whenever they belong to } \bar S
  \\ &&
  k_-(t,x) \leqslant k_-(t+1,x\pm1)\qquad \mbox{whenever they belong to } \bar S
  \\ &&
  k_+(t,x) = k_-(t+2,x)\qquad \mbox{whenever they belong to } \bar S
  \\ &&
  \mbox{For all } k=1,\dots,M \mbox{ there is } y\in S \mbox{ such that }
k_+(y)=k
  \\ &&
  \mbox{For all } k=0,\dots,M-1 \mbox{ there is } y\in S \mbox{ such that }
k_-(y)=k.
\end{eqnarray*}
To obtain $k$ from $k_\pm$, take, for $(t,x)\in S^*$,
\[
 k(t,x) = 
 \begin{cases}
   0 ,& t=-N-1 \\ 
   M ,& t=N+1 \\ 
   k_-(t+1,x)=k_+(t-1,x), & \mbox{otherwise}
 \end{cases}.
\]

We may also consider $t(j,x)$, $j \in \{1,\dots,M\}, x\in\{-N,\dots,N\}$.
For such $(j,x)$, take
\[
 t(j,x)=
 \begin{cases} 
 t, & (t,x)\in\bar S \mbox{ and } k_-(t,x)<j\leqslant k_+(t,x) \\ 
 -\infty,&  j\leqslant k_-(t,x) \mbox{ for all $t$ such that } (t,x)\in\bar S
\\ 
 +\infty,&  j > k_+(t,x) \mbox{ for all $t$ such that } (t,x)\in\bar S.
 \end{cases}
\]
When $|t(j,x)|\ne\infty$, $(t(j,x),x)$ is an element of $\bar S$.
Given a fixed $j$, if $t(j,x_0)=+\infty$ for some $x_0>0$ (resp. $<0$) then
$t(j,x)=+\infty$ for all $x\geqslant x_0$ (resp. $\leqslant x_0$); same for
$-\infty$. For $x=0$ it is always the case that $t(j,0)\in\{-N,\dots,N\}$.
Also, $t(j,x)\leqslant t(j+1,x)$ for $j\in\{1,\dots,M-1\}, x\in\{-N,\dots,N\}$.
For all $j\in\{1,\dots,M-1\}$ there is $x\in\{-N,\dots,N\}$ such that $t(j,x)<
t(j+1,x)$. Also $t(j,x+1) = t(j,x)\pm 1$ for all $j\in\{1,\dots,M\},
x\in\{-N,\dots,N-1\}$ such that $|t(j,x+1)|,|t(j,x)|\ne\infty$.
If $|t(j,x)|<\infty$ and $|t(j,x+1)|=\infty$ or $|t(j,x-1)|=\infty$ then
$(t(j,x),x)\in\partial\bar S$.

To relate $t(j,x)$ with $k_\pm(t,x)$ take $t(0,x)=-\infty$, $t(M+1,x)=+\infty$
and
\[
  k_-(t,x) = \max\{k:0\leqslant k\leqslant M, t(k,x)<t\},
\]
\[
  k_+(t,x) = \min\{k:0\leqslant k\leqslant M, t(k+1,x)>t\}.
\]

Now for $j\in\{1,\dots,M\}$ take $a = \min\{x:t(j,x)\ne\pm\infty\}, b =
\max\{x:t(j,x)\ne\pm\infty\}$ and define $ y_i = (t(j,x),x) $
for $a\leqslant x\leqslant b$.
Take $\ell_j = (y_a,e_{a+1},y_{a+1},\dots,e_b,y_b)$, $e_j=\langle
y_{j-1},y_j\rangle$. Then $\ell_1,\dots,\ell_M\in C(S)$ and
\begin{equation}
 \label{eq1ellconstrord}
 \ell_1\prec\ell_2\prec\cdots\prec\ell_M.
\end{equation}

Given the set $\ell_1\prec\cdots\prec\ell_M$, write
$D(\ell_j)=\{x^-_j,x^-_j+1,\dots,-1,0,1,\dots,x^+_j-1,x^+_j\}$ and
$y=(t_j(x),x)$ for $y\in\ell_j$.
Then $(t_j(x^\pm_j),x^\pm_j)\in\partial_\pm\bar S\cup\partial^\pm\bar S$. Take
\[
 t(j,x)=
 \begin{cases} 
    t_j(x), & x_j^-\leqslant x\leqslant x_j^+ \\
    +\infty, & x>x_j^+ \mbox{ and } t_j(x_j^+)>0 \mbox{ or }
               x<x_j^- \mbox{ and } t_j(x_j^-)>0 \\
    -\infty, & x>x_j^+ \mbox{ and } t_j(x_j^+)<0 \mbox{ or }
               x<x_j^- \mbox{ and } t_j(x_j^-)<0.
  \end{cases}
\]

This completes the set of equivalences:
\[
 \bar\eta \leftrightarrow p \leftrightarrow (A,k), \quad
  k \leftrightarrow k_\pm \leftrightarrow t \leftrightarrow \ell.
\]

For $y=(t,x)\in \bar S$, take $K(y)=\{k_-(y)+1,\dots,k_+(y)\}$. It follows from
the definition of $t(j,x)$ that
\begin{equation}
 \label{eq1tjequiv}
 \mbox{for all }(t,x)\in\bar S,j\in\{1,\dots,M\},
 \qquad j\in K(t,x) \Leftrightarrow t(j,x)=t.
\end{equation}

For $e\in\E(\bar S)$, take $k_-(e)=k_-(y) \vee k_-(y')$ and $k_+(e) = k_+(y)
\wedge k_+(y')$, where $e=\langle y,y' \rangle$. Also take
$K(e)=\{k_-(e)+1,\dots,k_+(e)\}=K(y)\cap K(y')$.
Yet for $\ell=\left( y_0,e_1,y_1,e_2,y_2,\dots,e_n,y_n \right)\subseteq\bar S$,
take $K(\ell)=K(y_0)\cap\cdots\cap K(y_n)=K(e_1)\cap\cdots\cap K(e_n)$.

\begin{claim}
 \label{claim11}
 Let $\ell\subseteq\bar S$ and $j\in\{1,\dots,M\}$. Then
 $\ell\subseteq\ell_j$ if and only if $j\in K(\ell)$.
\end{claim}

To prove the claim, take $\ell\subseteq\ell_j$. Write
$\ell=(y_{a'},e_{a'+1},y_{a'+1},\dots,y_{b'})$ and
$\ell_j=(y_{a},e_{a+1},y_{a+1},\dots,y_{b})$ with $a\leqslant a'<b'\leqslant b$.
By construction of $\ell_j$, for all $a\leqslant i\leqslant b$, $t(j,x_i)=t_i$
and, by~(\ref{eq1tjequiv}), $j\in K(y_i)$. Therefore $j\in K(a')\cap\cdots \cap
K(b') = K(\ell)$.
Now suppose $j\in K(\ell)$ for some $\ell\subseteq\bar S$. Write
$\ell=(y_0,e_1,y_1,\dots,y_n)$ and, for $i=0,\dots,n$, $y_i=(t_i,x_i)$. By
definition $j\in K(y_i)$ for all $i$, by~(\ref{eq1tjequiv}) this implies
$t(j,x_i)=t_i$ and by construction of $\ell_j$ we have $y_i\in\ell_j\ \forall\
i$; it follows from this last fact that $\ell\subseteq\ell_j$, thus proving the
claim.

Let $J^j=(q_{j-1},q_j],\ j\in\{1,\dots,M\}$. Notice that
\begin{equation}
 \label{eq1intrvldisjoint}
 J^j \cap J^{j'} = \emptyset\mbox{ when } j\ne j'.
\end{equation}
For $y\in S$, take $ J(y) = \bigcup_{j\in K(y)}J^j $ and for $\ell\subseteq\bar
S$ take $J(\ell) = \bigcup_{j\in K(\ell)}J^j$. Given $e=\langle
y,y'\rangle\in\E(\bar S)$, let $J(e)=(q_{k_-(e)},q_{k_+(e)}]$. Notice that
$J(e)=J(y)\cap J(y')$, since $J(e)=\cup_{j\in K(e)}J^j$. It follows
from~(\ref{eq1etap})-(\ref{eq1kconsist}) that
\[
 |J(e)| = \eta(e).
\]

By the above fact there is one translation $T_e$ that relates $(0,\eta(e)]$ with
$J(e)$. We associate the atoms and subintervals of these translated intervals
according to the following rules. We write $(e_1,p_1)\approx(e_2,p_2)$ when
$(e_1,T_{e_1}^{-1}p_1)\sim(e_2,T_{e_2}^{-1}p_2)$ and $(e_1,J_1)\approx(e_2,J_2)$
when $(e_1,T_{e_1}^{-1}J_1)\sim(e_2,T_{e_2}^{-1}J_2)$.

\begin{claim}
\label{claim1ltransequiv}
 Let $e_1=\langle y_0,y_1\rangle,e_2=\langle y_1,y_2\rangle$ with $x_0<x_1<x_2$
be two adjacent edges in $\E(\bar S)$ and $J_1,J_2$ intervals. Then
$(e_1,J_1)\approx(e_2,J_2)$ if and only if $J_1=J_2$ and $J_1\subseteq
J(y_0)\cap J(y_1)\cap J(y_2)$.
\end{claim}

To prove the claim it suffices to show that $(e_1,p_1)\approx(e_2,p_2)$ iff
$p_1=p_2\in J(e_1)\cap J(e_2)$, since $T_e:(0,\eta(e)]\to J(e)$ is just a
translation and $J(e_1)\cap J(e_2)=J(y_0)\cap J(y_1)\cap J(y_2)$.

So suppose $(e_1,p_1)\approx(e_2,p_2)$. Write $p_1=T_{e_1}\tilde p_1$ and
$p_2=T_{e_2}\tilde p_2$. Of course $\tilde p_1\in(0,\eta(e_1)]$, thus $p_1\in
J(e_1)$; also $p_2\in J(e_2)$. Writing $T_e$ more explicitly one gets $T_ep =
p+q_{k_-(e)}$. We have 4 cases to consider. Case~1: $e_1=e_{y_1}^\swarrow$,
$e_2=e_{y_1}^\nwarrow$.
In this case $\tilde p_1=\tilde p_2$, $k_-(y_0) \leqslant
k_-(y_1)$ and $k_-(y_2) \leqslant k_-(y_1)$, thus $k_-(e_1)=k_-(e_2)$ and
$p_1=\tilde p_1 + q_{k_-(e_1)} = \tilde p_2 + q_{k_-(e_2)} = p_2$, since by
hypothesis $\tilde p_1=\tilde p_2$.
Case~2: $e_1=e_{y_1}^\searrow$, $e_2=e_{y_1}^\nwarrow$. In this case $\tilde
p_1=\tilde p_2 - \eta(e_{y_1}^\swarrow)$ and $k_-(y_2)\leqslant k_-(y_1)$, thus
$k_2(e_2) = k_-(y_1)$. By~(\ref{eq1etap}) $q_{k_-(y_1)} + \eta(e_{y_1}^\swarrow)
= q_{k(t_1-1,x_1)}+\eta(e_{y_1}^\swarrow) = q_{k(t_1,x_1-1)} = q_{k_-(e_1)}$, so
$p_1=\tilde p_1 + q_{k_-(e_1)} = \tilde p_1 + q_{k_-(y_1)} +
\eta(e_{y_1}^\swarrow) = \tilde p_2 + q_{k_-(y_1)} = p_2$. Cases 3 and 4 are
similar.

Conversely, suppose $p_1=p_2\in J(e_1)\cap J(e_2)$. Consider Case~3, i.e.,
$e_1=e_{y_1}^\swarrow$, $e_2=e_{y_1}^\nearrow$; the other cases are similar.
Take $\tilde p_1= T_{e_1}^{-1}p_1 = p_1-q_{k_-(e_1)}= p_1-q_{k_-(y_1)} \in
(0,\eta(e_1)]$
and $\tilde p_2= T_{e_2}^{-1}p_2 = p_2-q_{k_-(e_2)}= p_2-q_{k_-(y_2)} =
p_2-q_{k_-(t_1+1,x_1+1)} = p_2 - p_{x_1+1,t_1} = p_2 -
[p_{x_1,t_1-1}+\eta(e_{y_1}^\nwarrow)] = p_2 - q_{k_-(y_1)} -
\eta(e_{y_1}^\nwarrow) = \tilde p_1-\eta(e_{y_1}^\nwarrow) \in
(0,\eta(e_2)]$.
Since $\tilde p_2>0$ we have $\tilde p_1>\eta(e_{y_1}^\nwarrow)$.
Therefore $(e_1,p_1)\approx(e_2,p_2)$ and the claim holds.

It will be convenient to work with a different representation of a broken line,
which we dub \emph{translated broken line}. The translated broken lines have a
simpler representation that follows from Claim~\ref{claim1ltransequiv}.
Given a broken line $\gamma\in B(\bar\eta)$ in $\bar S$, we define the object
$\psi=\psi(\gamma)$ by
\[
 \psi = \left((y_0,e_1,y_1,\dots,e_n,y_n),J\right),
\]
where $(y_0,e_1,y_1,\dots,e_n,y_n)=\ell(\gamma)$ and
$J=T_{e_1}J_1=T_{e_2}J_2=\cdots=T_{e_n}J_n$.
Write $\gamma(\psi)$ for the unique broken line $\gamma$ such that
$\psi=\psi(\gamma)$.
For any $\psi$ of the above form we have $\gamma(\psi) =
(y_0,e_1,J_1,y_1,\dots,e_n,J_n,y_n)$, where $J_i=T_{e_i}^{-1}J$.
The translated broken lines have all the properties analogous to those already
discussed for the broken lines.
For $\psi=(\ell_1,J)$, define $\ell(\psi)=\ell_1$ and
$w(\psi)=|J|=w(\gamma(\psi))$.

\begin{claim}
 \label{claim12}
 Let $\ell=(y_0,e_1,y_1,\dots,e_n,y_n)\subseteq\bar S$ and $J\subseteq(0,q_M]$.
 Then $(e_1,J)\approx\cdots\approx(e_n,J)$ if and only if $J\subseteq J(\ell)$.
\end{claim}
The proof is short. Suppose $(e_1,J)\approx\cdots\approx(e_n,J)$. Since
$(e_1,J)\approx(e_2,J)$, by Claim~\ref{claim1ltransequiv} we have $J\subseteq
J(y_0)\cap J(y_1)\cap J(y_2)$. Also $(e_2,J)\approx(e_3,J)$, thus $J\subseteq
J(y_3)$ as well. Similarly, for $i=2,\dots,n$, $(e_{i-1},J)\approx(e_i,J)$ and
by Claim~\ref{claim1ltransequiv} $J\subseteq J(y_i)$. Therefore
$J\subseteq\cap_{i=0}^n J(y_i) = \cap_{i=0}^n \cup_{j\in K(y_i)} J^j =
\cup_{j\in [\cap_{i=0}^n K(y_i)]} J^j = \cup_{j\in K(\ell)} J^j = J(\ell)$; we
have used~(\ref{eq1intrvldisjoint}) on the second equality. Conversely, if
$J\subseteq J(\ell)$ we have $J\subseteq J(y_i)$ for $i=0,\dots,n$ and by
Claim~\ref{claim1ltransequiv} we have $(e_{i-1},J)\approx(e_i,J)$ for all
$i=2,\dots,n$, i.e., $(e_1,J)\approx\cdots\approx(e_n,J)$, which proves the
claim.

We define $\tilde B(\bar\eta) = \{\psi(\gamma):\gamma\in B(\bar\eta)\}$. By
definition of translated broken lines we have $\psi\in\tilde B(\bar\eta)$ if and
only if $(e_1,J)\approx\cdots\approx(e_n,J)$. It follows from
Claim~\ref{claim12} that
\begin{equation}
 \label{eq1tblequiv}
 \tilde B(\bar\eta) = \{\psi=(\ell,J):\ell\subseteq\bar S, J\subseteq J(\ell)\}.
\end{equation}
As for the broken lines, given an $\ell\subseteq\bar S$, one can define the
maximal translated broken line $\psi(\ell)$ in $\tilde B(\bar\eta)$ that has
trace $\ell$.
Notice that $\psi(\ell)=\psi(\gamma(\ell))$ and
$\gamma(\ell)=\gamma(\psi(\ell))$. Now by~(\ref{eq1tblequiv}) one has
$\psi(\ell) = (\ell,J(\ell))$ and therefore
\begin{equation}
 \label{eq1constrwcalc}
 w(\ell) = \sum_j w_j \I_{j\in K(\ell)}.
\end{equation}

Let $\ell\in C(S)$. By Claim~\ref{claim11} and Item~\ref{item1ellextreme} of
Lemma~\ref{lemma1lorder} we have $j\in K(\ell)$ if and only if $\ell=\ell_j$.
Thus we have by (\ref{eq1ellconstrord})
\begin{equation}
 \label{eq1constrkell}
 K(\ell) =
 \begin{cases} 
    \{j\}, & \ell=\ell_j \\ 
    \emptyset, & \mbox{otherwise}
 \end{cases}
 \qquad \mbox{for any }\ell\in C(S).
\end{equation}

Now (\ref{eq1lbtcarac}) follows from (\ref{eq1constrkell}) and
(\ref{eq1constrwcalc}). Finally, combining (\ref{eq1constrwcalc}) and
Claim~\ref{claim11} gives
\[
 w(\ell) = \sum_j w_j \I_{\ell\subseteq\ell_j}
\]
and (\ref{eq1lweightcalc}) follows from (\ref{eq1lbtcarac}) and the above
equation.

It remains to prove that, given any pair of sets
$\{\ell_1\prec\ell_2\prec\cdots\prec\ell_M\}$ and $\{w_1,\dots,w_M>0\}$,
equation~(\ref{eq1lbtcarac}) holds for some $\bar\eta$. For $j=0,1,\dots,M$, let
$q_j=\sum_{j'=1}^jw_{j'}$ and let $A=\{0=q_0<q_1<\cdots<q_M\}$.
For $x\in\{-N,\dots,N\}$, $j\in\{1,\dots,M\}$ define $t(j,x)$ from the
$\{\ell_j\}$ as shown above.
It then follows that $t(j,x)$ has all the properties mentioned in the
construction. From $t(j,x)$, define $k_\pm(y),y\in\bar S$ and then $K(y),y\in
S^*$.
With $A$ defined above and $k(\cdot)$ we can recover $p_{t,x},(t,x)\in S^*$ and
from that we obtain the associated flow field $\bar\eta$. But for such
$\bar\eta$~(\ref{eq1lbtcarac}) holds by the construction described at this
appendix.

\section{Proof of~(\ref{eq1intersect})}
\label{sec1intersect}

As in the proof of Lemma~\ref{lemma1lorder}, take $A(t,x)=\{(\tilde t,\tilde
x):\tilde t\geqslant t+|\tilde x-x|\}$, $A(V)=\cup_{y\in V}A(y)$. Consider also
$\tilde A(t,x)=\{(\tilde t,\tilde x):\tilde t\leqslant t-|\tilde x-x|\}$. Then
$y\in A(y')$ iff $y'\in\tilde A(y)$.
Since $y_m\in A(\ell)$, we can take $n=\min\{i:y_i\in A(\ell)\}$.
There is $y_*\in\ell$ such that $y_n\in A(y_*)$
and thus $y_*\in\tilde A(y_n)$.
If $n=0$, $y_*\in\bar S\cap\tilde A(y_0)=\{(x_0,t_0),(x_0+1,t_0-1),(x_0-1,t_0-1)\}$ and any of these possibilities for $y_*$ imply $y_0\in\ell$.
So suppose $n>0$.
By construction $x_{n-1}=x_n\pm1$, assume for simplicity
$x_{n-1}=x_n+1$. Now $y_{n-1}\not\in A(\ell)$, which means $\ell\cap\tilde
A(y_{n-1})=\emptyset$. So $y_*\in\tilde A(y_n)\backslash\tilde A(y_{n-1})$ and
thus $x_*\leqslant x_n$, $t_*=t_n-(x_n-x_*)$. Notice that $\ell$ must eventually
reach $\{(t,x):t-(t_n+1)=(x_n+1)-x\}$ to cross $\bar S$. Let $y'=(t',x')$ be the
point of the first time it happens, that is, the one with smallest $x'$. Since
$\ell\cap\tilde A(y_{n-1})=\emptyset$, we must have $t'\geqslant t_n+1$, so
$x'\leqslant x_n+1$.
Thus $t'-t_* = (t'-t_n)+(t_n-t_*) \geqslant 1+x_n-x_* \geqslant x'-x*$ and the
equality holds if and only if $t'-t_n=1,x'-x_n=1$. But it must be the case that
equality holds because (\ref{eq1lbroktr}) implies $|t'-t_*|\leqslant|x'-x_*|$.
So $y'=(t_n+1,x_n+1)\in\ell$.
Now we only need to observe that $\ell$ must connect $y_*$ to $y'$ through
$y_n$. Indeed, if $t(x_n)>t_n$ we would have $t(x_n)-t(x_*)>x_n-x_*$ and if
$t(x_n)<t_n$ we would have $t(x')-t(x_n)>x'-x_n$ and either of them is absurd
because of~(\ref{eq1lbroktr}).

\section*{Acknowledgments}
We thank V.~Beffara and A.~Ramírez for fruitful discussions.
L.~T.~Rolla thanks the hospitality of PUC-Chile.
This work had financial support from CNPq grants 302796/2002-9, 141114/2004-5, 302221/2008-5, and 485071/2006-1, FAPERJ s.n., FAPERJ grant E-26/100.626/2007 APQ1, FAPESP grant 07/58470-1, FSM-Paris, and the Lithuanian State Science and Studies Foundation grant T-70/09.


\begin{thebibliography}{10}

\bibitem{aldous95}
{\sc D.~Aldous and P.~Diaconis}, {\em Hammersley's interacting particle process
  and longest increasing subsequences}, Probab. Theory Related Fields, 103
  (1995), pp.~199--213.

\bibitem{arak89}
{\sc T.~Arak and D.~Surgailis}, {\em {M}arkov fields with polygonal
  realizations}, Probab. Theory Related Fields, 80 (1989), pp.~543--579.

\bibitem{arak89b}
\leavevmode\vrule height 2pt depth -1.6pt width 23pt, {\em On polygonal
  {M}arkov fields}, in Stochastic methods in mathematics and physics, World
  Sci. Publ., 1989, pp.~302--309.

\bibitem{arak89a}
\leavevmode\vrule height 2pt depth -1.6pt width 23pt, {\em Polygonal fields: a
  new class of {M}arkov fields on the plane}, in Stochastic differential
  systems, vol.~126 of Lecture Notes in Control and Inform. Sci., Springer,
  Berlin, 1989, pp.~293--316.

\bibitem{arak89c}
\leavevmode\vrule height 2pt depth -1.6pt width 23pt, {\em Polygonal {M}arkov
  random fields}, Soviet Math. Dokl., 38 (1989), pp.~284--287.

\bibitem{arak90}
\leavevmode\vrule height 2pt depth -1.6pt width 23pt, {\em {M}arkov random
  graphs and polygonal fields with {Y}-shaped nodes}, Probability theory and
  mathematical statistics, 1 (1990), pp.~57--67.

\bibitem{arak91}
\leavevmode\vrule height 2pt depth -1.6pt width 23pt, {\em Consistent polygonal
  fields}, Probab. Theory Related Fields, 89 (1991), pp.~319--346.

\bibitem{balazs06}
{\sc M.~Bal{\'a}zs, E.~Cator, and T.~Sepp{\"a}l{\"a}inen}, {\em Cube root
  fluctuations for the corner growth model associated to the exclusion
  process}, Electron. J. Probab., 11 (2006), pp.~no. 42, 1094--1132
  (electronic).

\bibitem{fulton97}
{\sc W.~Fulton}, {\em Young tableaux}, vol.~35 of London Mathematical Society
  Student Texts, Cambridge University Press, Cambridge, 1997.
\newblock With applications to representation theory and geometry.

\bibitem{jockusch95}
{\sc W.~Jockusch, J.~Propp, and P.~Shor}, {\em Random domino tilings and the
  arctic circle theorem}.
\newblock arXiv:math/9801068, 1995.

\bibitem{johansson00}
{\sc K.~Johansson}, {\em Shape fluctuations and random matrices}, Comm. Math.
  Phys., 209 (2000), pp.~437--476.

\bibitem{oconnell00}
{\sc N.~O'Connell}, {\em Directed percolation and tandem queues}, HP Labs
  Technical Reports HPL-BRIMS-2000-28, Hewlett-Packard, 2000.
\newblock \url{http://www.hpl.hp.com/techreports/2000/HPL-BRIMS-2000-28.html}
  accessed April 2009.

\bibitem{pak06}
{\sc I.~Pak}, {\em Partition bijections, a survey}, Ramanujan J., 12 (2006),
  pp.~5--75.

\bibitem{rost81}
\leavevmode\vrule height 2pt depth -1.6pt width 23pt, {\em Nonequilibrium
  behaviour of a many particle process: density profile and local equilibria},
  Z. Wahrsch. Verw. Gebiete, 58 (1981), pp.~41--53.

\bibitem{rostXX}
{\sc H.~Rost}.
\newblock private communication.

\bibitem{seppalainen98}
{\sc T.~Sepp{\"a}l{\"a}inen}, {\em Coupling the totally asymmetric simple
  exclusion process with a moving interface}, Markov Process. Related Fields, 4
  (1998), pp.~593--628.

\bibitem{sidoravicius99}
{\sc V.~Sidoravicius, D.~Surgailis, and M.~E. Vares}, {\em Poisson broken lines
  process and its application to {B}ernoulli first passage percolation.}, Acta
  Appl. Math., 58 (1999), pp.~311--325.

\end{thebibliography}
\end{document}